\documentclass[reqno]{amsart}
\usepackage{graphicx} 

\usepackage{amsmath}%
\usepackage{amsthm}
\usepackage{amsfonts}%
\usepackage{amssymb}%
\usepackage{graphicx}
\usepackage{tikz-cd}
\usetikzlibrary{cd}
\usepackage{bm}
\usepackage{comment}
\usepackage{todonotes}

\usepackage{mathrsfs}

\DeclareFontFamily{OT1}{pzc}{}
\DeclareFontShape{OT1}{pzc}{m}{it}{<-> s * [1.10] pzcmi7t}{}
\DeclareMathAlphabet{\mathpzc}{OT1}{pzc}{m}{it}

\newtheorem{theorem}{Theorem}[section]

\newtheorem{corollary}[theorem]{Corollary}

\newtheorem{definition}[theorem]{Definition}

\newtheorem{lemma}[theorem]{Lemma}

\newtheorem{proposition}[theorem]{Proposition}
\newtheorem{remark}[theorem]{Remark}

\numberwithin{equation}{section}

\def\XXint#1#2#3{{\setbox0=\hbox{$#1{#2#3}{\int}$}
\vcenter{\hbox{$#2#3$}}\kern-.5\wd0}}

\newcommand{\R}{\mathbb{R}}

\newcommand{\N}{\mathbb{N}}
\newcommand{\Z}{\mathbb{Z}}

\newcommand{\Ha}{\mathcal{H}}
\newcommand{\leb}{\mathcal{L}}

\newcommand{\sgn}{\operatorname{sgn}}
\newcommand{\spt}{\operatorname{spt}}

\newcommand{\Mod}{\operatorname{Mod}}
\newcommand{\dist}{\operatorname{dist}}
\newcommand{\diam}{\operatorname{diam}}

\newcommand{\diag}{\operatorname{diag}}

\newcommand{\Lip}{\operatorname{Lip}}
\newcommand{\LIP}{\operatorname{LIP}}

\newcommand{\ud}{\mathrm {d}}

\newcommand{\inv}{^{-1}}

\DeclareMathOperator{\md}{\operatorname{md}}

\newcommand{\J}{\mathbf{J}}

\newcommand{\bb}[1]{\llbracket #1\rrbracket}
\usepackage{stmaryrd}

\usepackage{accents}
\newcommand{\minv}[1]{\accentset{\leftarrow}{#1}}

\title[Multi-valued inverse of a QR map is a QR curve]{Pull-back of differential forms by multi-valued Sobolev maps, and the quasiregularity of the multi-valued inverse of a quasiregular map}
\date{}

\author{Elefterios Soultanis}
\address{Department of mathematics and statistics\\ University of Jyv\"askyl\"a, Mattilanniemi (MaD), PL35, 40014}
\email{elefterios.e.soultanis@jyu.fi}


\begin{document}

\begin{abstract}
We use Almgren's framework of multi-valued maps to construct a multi-valued inverse $\minv f:f(\Omega)\to \mathcal A_d(\R^n)$ of a quasiregular map $f:\Omega\to \R^n$ of finite degree $d$. We then develop a pull-back theory of differential forms on $\mathcal A_d(\R^n)$ by Sobolev maps, and use it to show that the multi-valued inverse is a quasiregular $\omega$-curve (in the sense of Pankka) with respect to a natural $n$-form $\omega$ (suitably interpreted). The pull-back theory is of independent interest, and allows us to conclude e.g. higher Sobolev integrability and quasiminimality of the multi-valued inverse.
\end{abstract}

\maketitle

\section{Introduction}

\subsection{Overview}
A continuous map $f\in W^{1,n}_{loc}(M^n,N^k)$ between oriented Riemannian manifolds with $n\le k$ is a $K$-quasiregular $\omega$-curve, if 
\begin{align}\label{eq:QR-curve}
\|\omega\|\circ f\|Df\|^n\le K\star f^*\omega\quad \textrm{ almost everywhere}.
\end{align}
Here, $\omega$ is a (closed) $n$-form on $N$, $\star f^*\omega$ is the Hodge-star dual of $f^*\omega$, and $\|\omega\|$ is the pointwise comass of $\omega$. When $k=n$ and $\omega$ is the volume form on $N$, \eqref{eq:QR-curve} recovers the notion of quasiregular (QR) maps, written more classically as 
\begin{align}\label{eq:QR}
\|Df\|^n\le KJ_f\quad\textrm{ almost everywhere},
\end{align}
where $J_f=\det(Df)$ and $Df$ is the weak differential of $f$. Quasiregular maps have a rich theory \cite{ric93,HKM06}, and enjoy many of the same analytic properties as quasiconformal (QC) maps. While they are not necessarily homeomorphisms, a seminal result of Reshetnyak \cite{res67} states that QR maps are open and discrete, and thus in particular locally injective outside a singular branch set of topological dimension at most $n-2$ \cite{vai66}. We call an open discrete map a branched cover.

Quasiregular curves (QR curves) in turn were introduced by Pankka \cite{pankka20}, and they generalize quasiregular maps by allowing the dimension of the target to be higher than that of the domain; for example pseudoholomorhic curves \cite{gro85} are QR curves. Much of the classical theory of QR maps extends to QR curves \cite{pankka20,pan-onn21}, with one notable exception: Reshetnyak's theorem. Namely, QR curves need not be branched covers. Indeed, if $k>n$, a QR curve cannot be open, while the more non-trivial failure of discreteness follows from an example in \cite{IVV02}.

Compared to QR maps, one of the challenges in QR curves is understanding how the choice of $\omega$ reflects the behaviour of $f$, see for example \cite{hei-pan-pry23,iko-pan24} for some rigidity results and connections with calibrations. 

\bigskip\noindent In this paper, we identify a fruther connection between QR maps and QR curves: we construct a multi-valued inverse $\minv f$ for a QR map $f:\Omega\to \R^n$ of finite degree $d>0$, and show that $\minv f$ is naturally a QR curve (see Theorem \ref{thm:QR-curve}). In particular, QR maps have a local (multi-valued) inverse across the singular branch set (and not only outside it) with suitable Sobolev regularity. Earlier, related constructions include a push-forward operator \cite[p. 263]{HKM06}, \cite[Section 4]{teripekka} and a generalized local inverse \cite[Chapter II.5]{ric93}, \cite[Section 8]{onn-raj09}, both of which are important for obtaining modulus inequalities.

Our approach is based on Almgren's space of unordered tuples \cite{del11} (also known as the symmetric product, cf. \cite[Chapter 4.K]{hat02}), and brings a new perspective to quasiregular theory; together with a pull-back theory by multi-valued maps developed in this paper (see Theorems \ref{thm:sob-pull-back} and \ref{thm:lip-pullback-flat-form}) we obtain e.g. higher integrability of the generalized inverse \cite[Chapter II.5]{ric93} (see Corollary \ref{cor:QR-curve} and Theorem \ref{thm:higher-integrability}) as well as quasiminimality of the multi-valued inverse (see Theorem \ref{thm:quasimin}). We describe the construction and briefly discuss some its advantages over alternative approaches below.

\subsection{Multi-valued inverse}

Given a non-injective (proper) branched cover $f:\Omega\to \R^n$, there are several natural ways to interpret the set valued ``inverse'' map $f(\Omega)\ni y\mapsto f\inv(y)$. One is to consider the natural factorization $f=\bar f\circ P_f$, where $P_f:\Omega\to \widetilde\Omega_f:=\Omega/\sim$ is the projection onto the quotient space given by the equivalence relation $x\sim y$ if $f(x)=f(y)$, and $\bar f:\widetilde \Omega_f\to f(\Omega)$ the induced map, which in this case is a homeomorphism and thus invertible. Another approach is to view the set valued inverse as a map $f(\Omega)\to \Ha(\Omega)$ into the Hausdorff space of non-empty compact subsets of $\Omega$. A drawback of the first approach is that the target space of the inverse depends on $f$ and, more crucially, that $\bar f$ need not be geometrically quasiconformal unless $f$ is a BLD-map. The second approach, on the other hand, does not detect the multiplicities in the preimage, and $\Ha(\Omega)$ does not carry any natural differential forms, nor any bi-Lipschitz embedding in Euclidean space (cf. \cite[Theorem 2.1]{del11}).

Instead, we give a construction based on Almgren's space and local degree theory. Let $Q\ge 1$ be a natural number. The Almgren space $\mathcal A_Q(\R^n)$ of unordered $Q$-tuples in $\R^n$ was used by Almgren in his seminal work on regularity of minimal currents \cite{alm2000}, and was further explored by De Lellis et. al. \cite{del11}. It is a $nQ$-dimensional smooth orbifold, but not a manifold (unless $Q=1$), defined as
\begin{align*}
\mathcal A_Q(\R^n)=(\R^n)^Q/S_Q,\quad d_\mathcal A(\bb x,\bb y)=\inf_{\sigma\in S_Q}\Big(\sum_j^d|x_j-y_{\sigma(j)}|^2\Big)^{1/2},
\end{align*}
where $S_Q$ is the permutation group on $Q$ elements acting naturally on $(\R^n)^Q$. (See Section \ref{sec:almgren-space} for more details.) 

\bigskip\noindent The link between the multi-valued inverse and Almgren space is the following simple observation based on local degree theory. Given a proper branched cover $f:\Omega\to \R^n$ of degree $d>0$ from an open set $\Omega\subset \R^n$, the local index $\iota(f,x)$ is defined for all $x\in\Omega$ and satisfies $\displaystyle \sum_{x\in f\inv(y)}\iota(f,x)=d$, $y\in f(\Omega)$. Thus, $f\inv (y)$ can be regarded as an unordered $d$-tuple in $\Omega$ (counted with multiplicities) for each $y\in f(\Omega)$. The multi-valued inverse $\minv f$ of the proper branched cover $f:\Omega\to \R^n$ of finite degree $d>0$ is the map
\begin{align}\label{eq:minv}
\minv f:f(\Omega)\to \mathcal A_d(\R^n),\quad \minv f(y):= \sum_{x\in f\inv(y)}\iota(f,x)\bb x.
\end{align}
We refer the reader to Section \ref{sec:branched-covers} (see in particular Definition \ref{def:multi-valued-inv}) for the notation and further details. 

\subsection{Statement of results}

While it is not a manifold, $\mathcal A_Q(\R^n)$ nevertheless carries a natural $n$-form given by the trace of the volume form of $\R^n$; that is, the $n$-form 
\begin{align}\label{eq:natural-n-form}
\operatorname{tr}(\operatorname{vol}_{\R^n}):=\sum_j^Q P_j^*(\operatorname{vol}_{\R^n})
\end{align}
on $(\R^n)^Q$ (where $P_j$ is the projection to the $j$th factor) is $S_Q$-invariant: $\sigma^*\omega=\omega$ for all $\sigma\in S_Q$. Using $S_Q$-invariance, we develop a pull-back theory of multi-valued maps (see Sections \ref{sec:diff-forms} and \ref{sec:pull-back}), which is of independent interest.
\begin{theorem}\label{thm:sob-pull-back}
Let $n,m,Q\ge 1$ be natural numbers, $U\subset \R^m$ open, and $f\in W_{loc}^{1,p}(f(\Omega),\mathcal A_Q(\R^n))$. If $\omega$ is a $S_Q$-invariant smooth $k$-form on $(\R^n)^Q$, and either $k+1\le p$ of $\ud \omega=0$ and $k\le p$, then $\omega$ has a well-defined pull-back $f^*\omega \in L^{p/k}_{loc}(U,\bigwedge^k\R^m)$ and moreover
\begin{align*}
    \ud(f^*\omega)=f^*(\ud\omega)
\end{align*}
weakly.
\end{theorem}

In particular, if $f:U\to \mathcal A_Q(\R^n)$ is a multi-valued (locally) Lipschitz map, the pull-back $f^*\omega$ of any $S_Q$-invariant smooth $k$-form $\omega$ is a (locally) flat $k$-form on $U$ in the sense of Whitney (cf. Theorem \ref{thm:lip-pullback-flat-form}). 

Applying the definition of pull-back in the QR context we show that, if $f:\Omega\to \R^n$ is a quasiregular map, the multi-valued inverse is a quasiregular curve.

\begin{theorem}\label{thm:QR-curve}
Let $f:\Omega\to \R^n$ be a proper quasiregular map of finite degree $d$. Then the multivalued inverse $\minv{f}:f(\Omega)\to \mathcal A_d(\Omega)$ is an $\omega$-quasiregular curve, where $\omega=\operatorname{tr}(\operatorname{vol}_{\R^n})$ is given by \eqref{eq:natural-n-form}. More precisely,  $\minv f\in W_{loc}^{1,n}(f(\Omega),\mathcal A_d(\R^n))$, and 
\[
\|\omega\|\circ \minv f\ |D\minv f|^n\le d^{n/2-1}K_I(f)\star \minv f^*\omega\quad \textrm{almost everywhere on }\Omega.
\]
\end{theorem}

Here $K_I(f)$ is the inner distortion of the quasiregular map $f$, see Section \ref{sec:branched-covers}. Theorems \ref{thm:sob-pull-back} and \ref{thm:QR-curve} together with standard arguments found e.g. in \cite{pan-onn21} yield in particular the higher integrability of the generalized inverse of \cite[Chapter II.5]{ric93}. Recall that, given a proper quasiregular map $f:\Omega\to \R^n$, the generalized inverse is the map $g_f:f(\Omega)\to\R^n$ given by 
\begin{align}\label{eq:gen-inv}
g:f(\Omega)\to \R^n,\quad g(y):=\sum_{x\in f\inv(y)}\iota(f,x)x.
\end{align}
We record this higher integrability result in Corollary \ref{cor:QR-curve} below, along with the somewhat surprising fact that the generalized inverse \eqref{eq:gen-inv} is also a quasiminimizer of the Dirichlet energy. While both statements follow from the corresponding results for the multi-valued inverse (Theorems \ref{thm:higher-integrability} and \ref{thm:quasimin}), the latter requires further analysis of the structure of Almgren's space, presented in Appendix \ref{sec:appendix}. Indeed in Appendix \ref{sec:appendix} we show that $\mathcal A_d(\R^n)$ splits off a factor $\R^n$, and cannot split off any higher dimensional Euclidean factors, cf. Proposition \ref{prop:splitting}.
\begin{corollary}\label{cor:QR-curve}
Let $f:\Omega\to \R^n$ be a proper quasiregular map, and $g_f$ its generalized inverse. Then 
\begin{itemize}
    \item[(i)] there exists $p>n$ such that $g_f\in W^{1,p}_{loc}(f(\Omega),\R^n)$;
    \item[(ii)] if $h\in W^{1,n}_{loc}(f(\Omega),\R^n)$ is such that $\{h\ne g_f\}$ is essentially contained in a compact set $K\subset f(\Omega)$, then 
    \begin{align*}
        \int_K\|Dg_f\|^n\ud x\le d^{n/2-1}K_I\int_K\|Dh\|^n\ud x.
    \end{align*}
\end{itemize}
\end{corollary}

\bigskip\noindent The multi-valued inverse can also be seen as a homeomorphism $\minv f:f(\Omega)\to \Omega_f$ onto its image $\Omega_f:=\minv f(f(\Omega))\subset \mathcal A_d(\R^n)$ (Lemma \ref{lem:homeo}). Thus, although $\mathcal A_d(\R^n)$ is not a manifold, $\Omega_f$ is a metric $n$-manifold when equipped with the metric from $\mathcal A_d(\R^n)$. In the following theorem we obtain that $\minv f:f(\Omega)\to \Omega_f$ is geometrically quasiconformal.

\begin{theorem}\label{thm:geom-QC}
Let $f:\Omega\to \R^n$ be a proper quasiregular map of finite degree $d$, and equip $\Omega_f=\minv f(f(\Omega))$ with the metric from $\mathcal A_d(\R^n)$ and the Hausdorff $n$-measure on $\Omega_f$. 
\begin{itemize}
    \item[(1)] $\Omega_f$ is $n$-rectifiable, upper Ahlfors $n$-regular and satisfies the infinitesimal $n$-Poincar\'e inequality;
    \item[(2)] $\minv f:f(\Omega)\to \Omega_f$ is geometrically quasiconformal. More precisely, 
\begin{align*}
    \frac 1{K_IK_O}\Mod_n\Gamma\le \Mod_n\minv f(\Gamma)\le K_IK_O\Mod_n\Gamma
\end{align*}
fir any path family $\Gamma $ in $f(\Omega)$.
\end{itemize}
\end{theorem}

We refer to \eqref{eq:inf-n-PI} in Section \ref{sec:Omega_f-prop} for the definition of infinitesimal Poincar\'e inequality and further commentary. We mention here that the infinitesimal Poincar\'e inequality is not a quantitative condition.

\subsection{Main ideas and further discussion}
Theorem \ref{thm:sob-pull-back} follows by an approximation argument from its Lipschitz counterpart, Theorem \ref{thm:lip-pullback-flat-form}, whose proof uses induction, decomposition of multi-valued maps (cf. \cite[Proposition 1.6]{del11}), and a somewhat delicate approximation scheme. We outline the main idea.

In order to apply induction in the proof of Theorem \ref{thm:lip-pullback-flat-form}, we will need to approximate a given multi-valued Lipschitz map by maps which are locally decomposable. A given map $f:U\to \mathcal A_d(\R^n)$ fails to be locally decomposable on the set $F:=f\inv(\diag\mathcal A_d(\R^n))$. The restriction of $f$ to $F$ agrees with the map $d\bb{b(f)}$, where $b:\mathcal A_d(\R^n)\to \R^n$ is the barycenter map
\begin{align}\label{eq:barycenter}
b(\bb{x_1,\ldots,x_d})=\frac{x_1+\cdots+x_d}{d},
\end{align}
which is $\frac 1{\sqrt d}$-Lipschitz, see Corollary \ref{cor:barycenter}. The equality $f|_F=d\bb{b(f)}|_F$ implies that the differentials of the two functions agree almost everywhere on $F$. However, since $F$ is closed and not open, we need an approximation argument for the induction step to work. Our approximation scheme consists of a careful interpolation between $f$ and $d\bb{b(f)}$ in a neighbourhood of $F$.

The main step in the proof of Theorem \ref{thm:QR-curve} are the Sobolev estimates, which we obtain in Theorem \ref{thm:wug}. Our argument here uses Vitali's covering theorem, which is valid in high generality, and thus Theorem \ref{thm:wug} probably holds e.g. in the setting of \cite{onn-raj09}.

The proof of Theorem \ref{thm:geom-QC} employs a similar induction argument as in Theorem \ref{thm:lip-pullback-flat-form} together with some metric quasiconformal theory. We remark here that a large part of metric quasiconformal theory is based on the Ahlfors regularity of the spaces under consideration. Indeed, in our setting this property would imply the Loewner property of the image set $\Omega_f:=\minv f(f(\Omega)$ by a (deep) result of Semmes \cite{sem96}, which in turn implies the equivalence of the various notions of quasiconformality in the metric setting \cite{hei98}. In this paper, we establish the upper Ahlfors $n$-regularity of $\Omega_f$ (Proposition \ref{prop:ball-meas-upper-bound}), and an infinitesimal Poincar\'e inequality. It would be interesting to analyze the geometry of $\Omega_f$ further, in particular whether $\Omega_f$ admits local bi-Lipschitz parametrizations. By Almgren's bi-Lipschitz embedding $\mathcal A_d(\R^n)\hookrightarrow \R^N$, $\Omega_f$ can be regarded as a subset of Euclidean space, and the existence of bi-Lipschitz parametrizations is linked with the existence of Cartan--Whitney presentations in a suitable Sobolev class  \cite[[Theorem 1.2]{hei11}.

\subsubsection*{Acknowledgements} The research in this manuscript was supported by Research Council of Finland grant no. 355122. The author thanks Pekka Pankka, Ilmari Kangasniemi and Toni Ikonen for useful discussions on the manuscript.

\section{Preliminaries}

\subsection{Branched covers and quasiregular maps}\label{sec:branched-covers} For the notion of local degree and other topological preliminaries we refer the reader to \cite[Chapter I.4]{ric93}. In this paper we call a continuous open discrete map a \emph{branched cover}. A precompact open set $U\subset X$ is a normal neighbourhood of $f$ if $\partial f(U)=f(\partial U)$. A branched cover $f:X\to Y$ between topological $n$-manifolds has a well-defined local index $\iota(f,x)$ for every $x\in X$, satisfying
\begin{align*}
    \deg(f,U,y)=\sum_{x\in f\inv(y)\cap U}\iota(f,x).
\end{align*}
More generally, for any precompact open $U\subset X$, the map
\begin{align*}
f(U)\setminus f(\partial U)\to \Z,\quad y\mapsto \sum_{x\in f\inv(y)}\iota(f,x)
\end{align*}
is constant on each connected component of $f(U)\setminus f(\partial U)$. Moreover, $f$ is a local homeomorphism in some neighbourhood of a given point $x\in X$ if and only if $\iota(f,x)=\pm 1$. The branch set 
\[
B_f=\{x\in X: |\iota(f,x)|>1\}
\]
is a closed set of topological dimension at most $n-2$ \cite{vai66}. It follows that branched covers are either sense preserving ($\iota(f,\cdot)\ge 1$ everywhere) or sense reversing ($\iota(f,\cdot)\le -1$ everywhere). If $f:X\to Y$ is a proper branched cover of degree finite $d>0$, then
\begin{align}\label{eq:constant-deg}
\sum_{x\in f\inv(y)}\iota(f,x)=d\quad\textrm{for all }y\in f(X).
\end{align}
\begin{definition}\label{def:multi-valued-inv}
Let $f:X\to Y$ be a proper branched cover of finite degree $d>0$ between topological $n$-manifolds. The multi-valued inverse $\minv f:f(X)\to \mathcal A_d(X)$ of $f$ is defined by
\begin{align*}
\minv f(y)=\sum_{x\in f\inv(y)}\iota(f,x)\bb x    
\end{align*}
\end{definition}
It follows from \eqref{eq:constant-deg} that $\minv f$ is well-defined.

\medskip\noindent Recall that a Sobolev map $f\in W^{1,n}_{loc}(\Omega,\R^n)$ ($\Omega\subset\R^n$ open, $n\ge 2$) is called quasiregular, if there exists $K$ such that \eqref{eq:QR} holds. See \cite{ric93} for an excellent introduction to quasiregular maps. The least constant $K$ in \eqref{eq:QR} is called the outer distortion and denoted $K_O(f)$. The inner distortion $K_I(f)$ is the smallest constant $K$ such that 
\[
J_f \le K\min_{v\in S^{n-1}}\{ \|Df(v)\|^n\}\ \textrm{ almost everywhere}.
\]
If $f$ is quasiconformal, i.e. a quasiregular homeomorphism, then the inverse $f\inv $ is quasiconformal and $K_O(f\inv)=K_I(f)$, $K_I(f\inv)=K_O(f)$. We will make use of this fact without further mention in the sequel. 

By the seminal result of Reshetnyak \cite{res67}, non-constant quasiregular maps are continuous, open and discrete. In particular they are sense-preserving and $J_f\in (0,\infty)$ almost everywhere. 

\subsection{Space of unordered tuples}\label{sec:almgren-space} Let $X$ be a metric space and $d\ge 1$ a natural number. The Almgren space -- or symmetric product -- $\mathcal A_d(X)$ is the space $X^d/S_d$ equipped with the metric 
\begin{align}\label{eq:almgren-metric}
    d_\mathcal A(\bar x,\bar y)=\min_{\sigma\in S_d}\Big(\sum_j^dd(x_j,y_{\sigma(j)})^2\Big)^{1/2},
\end{align}
for $\bar x =\bb{x_1,\ldots,x_d},\ \bar y=\bb{y_1,\ldots,y_d}$. Here, the action by isometries on $X^d$ of the symmetric group $S_d$ is given by $\sigma\cdot(x_1,\ldots,x_d)=(x_{\sigma\inv(1)},\ldots,x_{\sigma\inv(d)})$. The space $\mathcal A_d(X)$ can alternatively be described as a subset of the Wasserstein space $\mathcal P_2(X)$, consisting of probability measures $\mu\in \mathcal P_2(X)$ such that $\#\spt(\mu)\le d$ and $d\cdot\mu(\{x\})\in\N$ for all $x\in X$ (note however that the metric inherited from $\mathcal P_2(X)$ differs from \eqref{eq:almgren-metric} by a factor of $1/\sqrt d$). Following \cite{del11}, we denote elements of $\mathcal{A}_d(X)$ either by $\bb{x_1,\ldots,x_d}$ or by $\displaystyle \sum_j^ma_j\bb{x_j}$, where the numbers $a_j>0$ satisfy $\sum_j^ma_j=d$.

\subsubsection*{Multi-valued differentiability} The following Rademacher's theorem for multi-valued Lipschitz maps is from \cite[Theorem 1.13]{del11}. In the statement $m,n,d\ge 1$ are natural numbers and $U\subset \R^m$ an open set.

\begin{theorem}\label{thm:multi-valued-rademacher}
Let $f:U\to \mathcal A_d(\R^n)$ be locally Lipschitz. Then $f$ is differentiable for almost every $x_0\in U$ in the following sense: there are linear maps $L_1,\ldots, L_d:\R^m\to \R^n$ so that 
\begin{align*}
    T_{x_0}f(x):=\bb{ f_1(x_0)+L_1(x-x_0),\ldots, f_d(x_0)+L_d(x-x_0) }
\end{align*}
satisfies $T_{x_0}f(x_0)=\bb{f_1(x_0),\ldots,f_d(x_0)}=f(x_0)$ and
\begin{itemize}
    \item[(i)] $\displaystyle \lim_{x\to x_0}\frac{d_{\mathcal A_d}(f(x),T_{x_0}f(x))}{|x-x_0|}=0$, and
    \item[(ii)] $L_i=L_j$ if $f_i(x_0)=f_j(x_0)$.
\end{itemize}
\end{theorem}
The differential $Df:=\bb{L_1,\ldots,L_d}$ can be regarded as an element in $\mathcal A_d(\R^{n\times n})$, but note that this does not determine it unambiguously, cf. \cite[Remark 1.11]{del11}. If $f_i:U\to \mathcal A_{d_i}(\R^n)$, $i=0,1$, are differentiable at $x\in U$ and $d_0+d_1=d$, then the map $f:=\bb{f_0,f_1}:U\to \mathcal A_d(\R^n)$ is differentiable at $x\in U$, and $D_xf=\bb{D_xf_0,D_xf_1}$. 

\subsubsection*{Multi-valued Sobolev maps} We refer the reader to \cite[Chapter 2.2 and Chapter 4]{del11} for the theory of Sobolev maps $W^{1,p}_{loc}(U,\mathcal A_d(\R^n))$ and mention here that Sobolev maps are approximately differentiable almost everywhere \cite[Corollary 2.7]{del11}. This follows from the following Lusin-type approximations by Lipschitz maps \cite[Proposition 2.5]{del11}.

\begin{proposition}\label{prop:lusin-approx}
    Let $f\in W^{1,p}(U,\mathcal A_d(\R^n))$. For every $\lambda>0$ there exists a $\lambda$-Lipschitz map $f_\lambda:U\to \mathcal A_d(\R^n)$ such that 
\begin{align}\label{eq:lusin-est}
\lambda^p\leb^m(\{f\ne f_\lambda\})\le C\int_{\{f\ne f_\lambda\}}(d_\mathcal A(f,d\bb 0)^p+|Df|^p)\ud x.
\end{align}
\end{proposition}
\begin{proof}
The claim is stated with $\{f\ne f_\lambda\}$ replaced by $U$ on the right hand side of \eqref{eq:lusin-est} in \cite[Proposition 2.5]{del11}, but the proof and references therein yield the inequality stated here.
\end{proof}

\subsubsection*{Metric Jacobian} Let $s$ be a seminorm on $\R^n$. The Jacobian $\J(s)$ is the unique number satisfying $\Ha^n_s=\J(s)\leb^n$, where $\Ha^n_s$ is zero if $s$ is degenerate, and the Hausdorff $n$-measure with respect to the metric induced by $s$ otherwise. 
If $f:U\to \mathcal A_d(\R^n)$ is (approximately) differentiable at $x\in U$, we denote 
\begin{align*}
\J f(x)=\J(\operatorname{md}_xf),
\end{align*}
where $\operatorname{md}_xf$ is the seminorm on $\R^n$ given by $\operatorname{md}_xf(v)=|D_xf(v)|$ for $v\in \R^n$.\footnote{Note that this is well-defined since the norm of $D_xf(v)$ is independent of the permutation of order of the elements in $D_xf=(D_xf_1,\ldots,D_xf_d)$.} We refer to \cite{kir94} where the metric differential and Jacobian were introduced in order to prove an area formula for maps into metric spaces (see also \cite{kar07}). We record the following identity for future use.

\begin{lemma}\label{lem:metric-jacobian}
Suppose  $L=\bb{L_1,\ldots,L_d}:\R^n\to \mathcal A_d(\R^n)$, where $L_1,\ldots, L_d:\R^n\to \R^n$ are linear maps, and let $s$ be the seminorm given by $s(v)=\Big(\sum_j^d|L_j(v)|^2\Big)^{1/2}$ for $v\in \R^n$. Then 
\[
\J(s)=\det(L_1^TL_1+\cdots+L_d^TL_d)^{1/2}.
\]
In particular, if $f:U\to \mathcal A_d(\R^n)$ is differentiable a.e., then 
\begin{align*}
\J f=\det\Big(\sum_j^dDf_j^TDf_j\Big).
\end{align*}
\end{lemma}
\begin{proof}
If $f$ is differentiable almost everywhere, then it is metrically differentiable almost everywhere and 
\[
\md_xf(v)=\Big(\sum_j^d|Df_j(v)|^2\Big)^{1/2},\quad v\in \R^m
\]
for a.e. $x\in U$. Thus the second claim readily follows from the first.

To prove the first, it suffices to consider the linear map $L=(L_1,\ldots,L_d):\R^n\to (\R^n)^d$. By the area formula on Euclidean spaces, we have that $J_{n}(L)L_\ast\leb^n=\Ha^n|_{\operatorname{Im}(L)}$, where  
\begin{align*}
J_n(L)=|\det(L^TL)|^{1/2}=|\det(L_1^TL_1+\cdots+L_d^TL_d)|^{1/2}
\end{align*}
Now $J_n(L)\ne 0$ if and only if $L$ has full rank, in which case $s$ is a genuine norm. Note that $L:(\R^n,s)\to (\operatorname{Im}(L),\|\cdot\|_{Eucl})$ is an isometry, and thus $L_\ast\Ha^n_s=\Ha^n|_{\operatorname{Im}(L)}$. Since $L:\R^n\to \operatorname{Im}(L)$ is bijective, this yields $\Ha^n_s=J_n(L)\leb^n$.
\end{proof}

\section{Multivalued inverse}

\subsection{Basic properties for branched covers} Throughout this subsection, $f:X\to Y$ is a proper branched cover of degree $d$ between metric spaces $X,Y$ homeomorphic to topological $n$-manifolds, $n\ge 2$. We start with a simple lemma. 
\begin{lemma}\label{lem:homeo}
The generalized inverse $\minv f:f(X)\to \mathcal A_d(X)$ is a homeomorphism onto its image.
\end{lemma}
In the proof we use the following notation. Let $U_f(x,r)$ be the connected component of $f\inv B(f(x),r)$ containing $x$. For each $y\in f(X)$, there exists $r_0$ such that $U_f(x,r)$ is a normal neighbourhood of $x$ for all $x\in f\inv(y)$ and $r\le r_0$. In particular the closure of $U_f(x,r)$ is compact and  $f(\partial U_f(x,r))=\partial B(f(x),r)$. 

\begin{proof}
It is clear that $\minv f$ is injective. To prove continuity, observe that 
\begin{align*}
\lim_{r\to 0}\diam U_f(x,r)=0
\end{align*}
for all $x\in X$. Indeed, 
\begin{align*}
\bigcap_j\overline U_f(x,2^{-j})=\overline U_f(x,r_0)\cap f\inv\Big(\bigcap_j\overline B(f(x),2^{-j})\Big)=\{x\},
\end{align*}
and since the sets $\overline U_f(x,2^{-j})$ are compact and nested for large enough $j$, it follows that the diameters must tend to zero. (See also \cite[Chapter I, Lemma 4.9]{ric93}.)

Fix $y\in f(X)$ and $\varepsilon>0$. Let $\delta>0$ be small enough such that $\diam U_f(x,\delta)<\varepsilon$ for all $x\in f\inv(y)$ (note that $\#f\inv(y)\le d$). For $y'\in B(y,\delta)$ we may group the preimages and obtain 
\[
f\inv(y')=\sum_{x\in f\inv(y)}\sum_{x'\in f\inv(y')\cap U_x}\iota(f,x')\delta_{x'}
\]
(where $U_x:=U_f(x,\delta)$). Moreover $\iota(f,x)=\sum_{x'\in f\inv(y')\cap U_x}\iota(f,x')$. Consequently
\begin{align*}
d_{\mathcal A_d}(f\inv(y),f\inv(y'))^2 &\le \sum_{x\in f\inv(y)}\sum_{x'\in f\inv(y')\cap U_x}\iota(f,x')d(x,x')^2\\
&\le \sum_{x\in f\inv(y)}\iota(f,x)\varepsilon^2=d\varepsilon^2.
\end{align*}
This proves the claimed continuity. We complete the proof by showing $\minv f\inv:\minv f(f(X))\to f(X)$ is continuous. Suppose $z_k=\minv f(y_k)\to z=\minv f(y)$. For small $r >0$ we have $f\inv B(y,r)=\bigcup_{x\in f\inv(y)}U_f(x,r)$ where each $U_f(x,r)=:U_x$ is a normal neighbourhood of $x$. For each $k$ write $f\inv(y_k)=\bigcup_{x\in f\inv(y)}(f\inv(y_k)\cap U_x)$. Since $\minv f(y_n)\to \minv f(y)$, for large $k$ we have that $f\inv(y_k)\cap U_x\ne \varnothing$ for all $x\in f\inv(y)$ and 
\[
d_\mathcal A(z_k,z)^2=\sum_{x\in f\inv(y)}\sum_{x_k\in f\inv(y_k)\cap U_x}d(x_k,x)^2\stackrel{k\to\infty}{\longrightarrow} 0.
\]
From this it follows that for each $x\in f\inv(y)$ there exists $x_k\in f\inv(y_k)$ with $x_k\to x$, which in turn implies $y_k=f(x_k)\to f(x)=y$ by the continuity of $f$. This proves that $\minv f\inv$ is continuous. 
\end{proof}

The next lemma relates the singular set of $\mathcal A_d(X)$ and the branch set of $f$. Here, the singular set $\operatorname{sing}\mathcal A_d(X)$ is the union of the sets
\[
\operatorname{sing}_k\mathcal A_d(X):=\{ \bb{x_1,\ldots,x_d} : \ x_1=\cdots =x_k\},\quad k=2,\ldots,d.
\]
We also denote $\operatorname{sing}_d\mathcal A_d(X)=\diag\mathcal A_d(X)$.

\begin{lemma}We have that 
\[
\minv f\inv(\operatorname{sing}(\mathcal A_d(X)))=f(B_f).
\]
\end{lemma}

\begin{proof}
A point $y\in f(X)$ is in $f(B_f)$ if and only if $f\inv(y)\cap B_f\ne \varnothing$. On the other hand, $F(y)\in \operatorname{sing}(\mathcal A_d(X))$ if and only if $\iota(f,x)>1$ for some $x\in f\inv(y)$. These two conditions are equivalent.
\end{proof}

The following proposition will be useful later on. It is based on a theorem of Floyd \cite{flo50}, see also \cite[Proposition 2.6]{lui17}, which states that compact paths admit lifts by a proper branched cover.
\begin{proposition}\label{prop:lifts}
Let $\gamma:[0,1]\to f(X)$ be a continuous path. Then 
\begin{align}\label{eq:lifts}
\minv f\circ\gamma=\sum_{x\in f\inv(\gamma_0)}\sum_j^{\iota(f,x)}\bb{\gamma_x^j},
\end{align}
where $\gamma_x^j$ ($j=1,\ldots,\iota(f,x)$) are the lifts of $\gamma$ by $f$ starting at $x$, i.e. $f\circ\gamma_x^j=\gamma$, $\gamma_x^j(0)=x$.
\end{proposition}
\begin{proof}
By arguing as in \cite[Chapter II, Corollary 3.4]{ric93} using \cite{flo50} (cf. \cite[Proposition 2.6]{lui17}), we see that $\gamma$ has $d$ total $f$-lifts $\tilde\gamma_1,\ldots,\tilde\gamma_d$ which satisfy $\#\{j:\ \tilde\gamma_j(t)=x\}=\iota(f,x)$ for all $x\in f\inv(\gamma_t)$. Relabeling these as $\gamma_x^j$ for $j=1,\ldots,\iota(f,x)$ and each $x\in f\inv(\gamma_0)$, this implies $\minv f(\gamma_t)=\sum_{x\in f\inv(y)}\sum_j^{\iota(f,x)}\bb{\gamma_x^j(t)}$ for all $t\in [0,1]$, proving \eqref{eq:lifts}. 
\end{proof}

\begin{proposition}\label{prop:hausd-meas-est}
Suppose $s\ge 0$ and $E\subset \minv f(f(X))$. Then
\[
\Ha^s((\minv f\circ f)\inv E)\le d\cdot \Ha^s(E).
\]
\end{proposition}
\begin{proof}
Denote $F=\minv f\circ f:X\to \mathcal A_d(X)$ and $X_f=F(X)$. Note that $F:X\to X_f$ is a branched cover of degree $d$. First suppose $E$ is contained in a compact set $K\subset X_f$. For each $z_0\in K$ let $r_0$ be so small that 
\[
F\inv B(z_0,r_0)=\bigcup_{x_0\in F\inv(z_0)}U_{F}(x_0,r_0)
\]
where $U_{F}(x_0,r_0)$ is a normal neighbourhood of $x_0$ for each $x_0\in F\inv(z_0)$. 
Suppose $\{E_i\}$ is a $\delta$-cover of $E$, where $\delta$ is smaller than half the Lebesgue number of the cover $\{B(z_0,r_0)\}_{z_0\in K}$. Then $E_i\subset B(z_i,r_i)$ for some $z_i\in K$ with $m=m_i$ preimages $x_1,\ldots,x_m$ under $F$ ($m\le d$). Thus
\[
F\inv(E_i)=\bigcup_{k}^mE_{i,k}, \quad E_{i,k}=F\inv(E_i)\cap U_{F}(x_k,r_i).
\]
Moreover, since $d_\mathcal A(F(x),F(y))\ge d(x,y)$ for $x,y\in U_{F}(x_k,r_i)$, it follows that $\diam E_{i,k}\le \diam E_i$ for all $i$ and $k=1,\ldots,m_i$. The two observations above imply that $\{E_{i,k}\}$ is a $\delta$-cover of $F\inv E$, yielding
\begin{align*}
\Ha_\delta^s(F\inv E)\le \sum_i\sum_{k=1}^{m_i}(\diam E_{i,k})^s\le \sum_i d\cdot(\diam E_i)^s.
\end{align*}
Taking infimum over all $\delta$-covers and sending $\delta$ to zero we obtain $$\Ha^s(F\inv E)\le d\Ha^s(E).$$ Exhausting $X_f$ by compact sets and using basic convergence theorems of measures we the claim follows for general $E\subset X_f$. 
\end{proof}

\subsubsection*{Push-forward} We define a push-forward operator $f_\ast$  as follows: given a Borel function $g:X\to [-\infty,\infty]$, set
\begin{align}\label{eq:push-fwd}
f_\ast g(y)=\sum_{x\in f\inv(y)}\iota(f,x)g(x),\quad y\in f(X).
\end{align}
We record here the fact that the push-forward preserves continuity and measurability. This should be compared to \cite[Lemma 14.31]{HKM06} and \cite[Lemma 4.1]{teripekka} where it is shown that it preserves Sobolev regularity (resp. Lipschitz regularity when $f$ is BLD). Later on the push-forward will prove useful in connection with area formulae, see Lemma \ref{lem:co-area} and Proposition \ref{prop:area-formula}.
\begin{lemma}
If $g\in C(X)$, then $f_\ast g\in C(f(X))$. If $g:X\to \R\cup\{\infty\}$ is Borel, then $f_\ast g:f(X)\to \R\cup\{\infty\}$ is Borel.
\end{lemma}
\begin{proof}
Given a function $g$ on $X$, define $g_{\mathcal A_d}$ on $\mathcal A_d(X)$ by 
\[
g_{\mathcal A_d}(\bb{x_1,\ldots,x_d})=g(x_1)+\cdots+g(x_d)
\]
Observe that $g_{\mathcal A_d}$ is continuous (resp. Borel) if $g$ is continuous (resp. Borel). Finally note that $f_\ast g=g_{\mathcal A_d}\circ \minv f$, from which both claims follow.
\end{proof}

\subsection{Sobolev estimates}
In this subsection we specialize to quasiregular maps on Euclidean domains. Throughout this subsection we assume that $f:\Omega\to \R^n$ is a proper quasiregular map of degree $d$, where $\Omega\subset\R^n$ is a domain and $n\ge 2$. The following lemma is a reformulation of the area formula using the push-forward operator. Together with Theorem \ref{thm:wug} below it yields the local Sobolev regularity of the multi-valued inverse.

\begin{lemma}\label{lem:co-area}
If $g:X\to [0,\infty]$ is Borel, then 
\begin{align*}
\int_{\Omega'} f_\ast g(y)\ud\Ha^n(y)=\int_\Omega g\J f \ud\Ha^n.
\end{align*}
\end{lemma}
\begin{proof}
Since $|B_f|=0$, we have that $\iota(f,x)=1$ $\Ha^n$-a.e. $x\in\Omega$. Recall that a quasiregular map has the property (N) and thus the co-area formula holds for $f$. It follows that, for $\Ha^n$-a.e. $y\in \Omega'$ we have that $f\inv(y)\cap B_f=\varnothing$. The co-area formula now yields
\begin{align*}
\int_{\Omega'}f_\ast g\ud\Ha^n=\int_{\Omega'}\sum_{x\in f\inv(y)}g(x)\ud\Ha^n(y)=\int_\Omega g\J f\ud\Ha^n,
\end{align*}
as claimed.
\end{proof}

The main result of this subsection is the following curve-wise estimates for the multi-valued inverse. In the statement $K_I=K_I(f)$ and $K_O=K_O(f)$.
\begin{theorem}\label{thm:wug}
For $\Mod_n$-a.e. curve $\gamma$ in $f(\Omega)$, $\minv f\circ\gamma$ is an absolutely continuous curve in $\mathcal{A}_d(\Omega)$, and satisfies
\begin{align}\label{eq:ug-ineq}
H(\gamma_t)|\gamma_t'|\le |(\minv f\circ\gamma)'_t|\le (K_IK_O)^{1/n} H(\gamma_t)|\gamma_t'|
\end{align}
a.e. $t$, where
\begin{align}\label{eq:H}
    H(y):=\big(f_\ast(\|Df\|^{-2})(y)\big)^{1/2},\quad y\in f(\Omega).
\end{align}
\end{theorem}
\begin{proof}
If $\Gamma_{not}$ is the family of paths in $\Omega$ which are not locally absolutely continuous, then $\Mod_n\Gamma_{not}=0$. By Poletski's inequality \cite[Chapter II, Theorem 8.1]{ric93} we have that $\Mod_nf(\Gamma_{not})\le K_I\Mod_n\Gamma_{not}=0$. Thus, for $\Mod_n$-a.e. $\gamma$ in $\Omega'$, every lift $\tilde\gamma$ of $\gamma$ is locally absolutely continuous. By Proposition \ref{prop:lifts} we have $\minv f\circ\gamma=\bb{\tilde\gamma_1,\ldots,\tilde\gamma_d}$, where $\tilde \gamma_j$ are the lifts of $\gamma$ by $f$ (counted with multiplicity). It follows that
\[
d_\mathcal A(\minv f(\gamma_t),\minv f(\gamma_s))^2\le \sum_l^d|\tilde\gamma_l(t)-\tilde\gamma_l(s)|^2,\quad s,t\in [0,1]
\]
which implies that $\minv f\circ\gamma$ is absolutely continuous whenever each lift $\tilde\gamma_l$ is absolutely continuous. In particular, $\minv f\circ\gamma$ is absolutely continuous for $\Mod_n$-a.e. $\gamma$. 

We now prove \eqref{eq:ug-ineq}. Let $\Gamma_0$ be a path family in $\Omega$ with $\Mod_n \Gamma_0=0=\Mod_nf(\Gamma_0)$ so that $\|Df\|$ is an upper gradient of $f$ along every $\gamma\notin \Gamma_0$. Observe that $f(B_f)$ is closed since $f$ has finite degree and that $|f(B_f)|=0$ by \cite[Chapter II, Theorem 7.4]{ric93}. If $B=B(y,r)\subset\Omega'\setminus f (B_f)$ is a ball such that $f\inv B=U_1\cup\cdots\cup U_d$, $U_l$ and $f:U_l\to B$ is quasiconformal, denote by $g_l:B\to U_l$ the inverse of $f|_{U_l}$, for $l=1,\ldots, d$. Then
\begin{align}\label{eq:inv-comp}
1\le \|Df(g_l(y))\|^n\|Dg_l(y)\|^n\le K_O\J f(g_l(y))K_I\J g(y)=K_OK_I
\end{align}
a.e. $y\in B$, see \cite{vai71}. By Vitali's covering theorem there exist countably many disjoint balls $B_j=B(y_j,r_j)$ as above, with inverses $g^j_l:B_j\to U_l^j$ satisfying \eqref{eq:inv-comp} outside a null set $N_j\subset B_j$, such that $\big|[\Omega'\setminus f(B_f)]\setminus \bigcup_jB_j\big|=0$. Denote $E=f(B_f)\cup\Omega'\setminus\bigcup_jB_j$, so that $|E|=0$. For each $j$, let $\Gamma^j$ be a path-family in $B_j$ with $\Mod_n\Gamma^j=0$ such that $\|Dg_l^j\|$ is an upper gradient along $\gamma$, and $|\gamma\inv(N_j)|=0$, for every $\gamma\notin \Gamma^j$. Let $\Gamma_1$ by the family of paths in $\Omega'$ which have a subcurve in some $\Gamma^j$. Then $\Mod_n\Gamma_1=0$. Finally, define $\Gamma_2=\Gamma_1\cup f(\Gamma_0)\cup\Gamma_{E}^+$ and note that $\Mod_n\Gamma_2=0$.

Let $\gamma\notin\Gamma_2$. Then $[0,1]$ can be covered up to a null-set by open sets $\gamma\inv(B_j)$, and for a.e. $t\in \gamma\inv(B_j)$ we have 
\begin{align*}
|(g^j_l\circ\gamma)_t'|&\le \|Dg^j_l(\gamma_t)\||\gamma_t'| \\
|\gamma_t'|=|(f\circ g_l^j\circ\gamma)_t'|&\le \|Df(g_l^j(\gamma_t))\||(g_l^j\circ\gamma)_t'|\\
\frac 1{\|Df(g_l^j(\gamma_t))\|}&\le \|Dg^j_l(\gamma_t)\|\le \frac {(K_IK_O)^{1/n}}{\|Df(g_l^j(\gamma_t))\|}.
\end{align*}
Thus 
\[
\frac{|\gamma_t'|}{\|Df(g_l^j(\gamma_t))\|}\le |(g_l^j\circ\gamma)_t'|\le \frac{(K_IK_O)^{1/n}|\gamma_t'|}{\|Df(g_l^j(\gamma_t))\|}\quad \textrm{a.e. }t\in \gamma\inv(B_j)
\]
Since $ F\circ\gamma =\bb{g_1^j\circ\gamma,\ldots,g_d^j\circ\gamma }$ on $\gamma\inv(B_j)$, we obtain
\begin{align*}
|(\minv f\circ\gamma)_t'|^2&=\sum_{l}^d|(g_l^j\circ\gamma)'_t|^2\le (K_IK_O)^{2/n}\sum_l^d\frac{|\gamma_t'|^2}{\|Df(g_l^j(\gamma_t))\|^2}\\
&=(K_IK_O)^{2/n}f_\ast(\|Df\|^{-2})(\gamma_t)|\gamma_t'|^2,
\end{align*}
and similarly
\begin{align*}
|(\minv f\circ\gamma)_t'|^2\ge f_\ast(\|Df\|^{-2})(\gamma_t)|\gamma_t'|^2
\end{align*}
a.e. $t\in \gamma\inv(B_j)$ for every $j$. This proves \eqref{eq:ug-ineq} and finishes the proof of the Proposition.
\end{proof}

\begin{corollary}\label{cor:sob-reg}
The multivalued inverse $\minv f:f(\Omega)\to \mathcal A_d(\Omega)$ is Sobolev: $\minv f\in W^{1,n}_{loc}(f(\Omega),\mathcal A_d(\Omega))$. More precisely, the minimal $n$-weak upper gradient $|D\minv f|$ of $\minv f$ satisfies 
\[
H\le |D\minv f|\le (K_IK_O)^{1/n}H
\]
and 
\begin{align*}
\int_U|D\minv f|^n\ud y\le d^{n/2-1}K_IK_O|f\inv U|
\end{align*}
for any compact $U\subset f(\Omega)$.
\end{corollary}
\begin{proof}
By Theorem \ref{thm:wug} (and using its notation), $(K_IK_O)^{1/n}H$ is a weak upper gradient of $F$. Thus $|D\minv f|\le (K_IK_O)^{1/n}H$. Using e.g. \cite[Theorem 1.1]{teri-syl24} we obtain $H\le |D\minv f|$ from the left-hand inequality in \eqref{eq:ug-ineq}. For any compact $U\subset f(\Omega)$ we have, by Lemma \ref{lem:co-area}, that
\begin{align*}
\int_UH^n\ud y&\le d^{n/2-1}\int_Uf_\ast(\|Df\|^{-n})\ud y\\
&=d^{n/2-1}\int_{f\inv U}\frac{\J f}{\|Df\|^n}\ud x\le d^{n/2-1}|f\inv U|.
\end{align*}
This completes the proof of the claim.
\end{proof}

\subsection{Area formula}
We continue assuming $f:\Omega\to \R^n$ is a non-constant quasiregular map of degree $d$. We will prove that the multi-valued inverse satisfies Lusin's properties $(N)$ and $(N\inv)$ as well as the area formula.
\begin{proposition}\label{prop:area-formula}
The multi-valued inverse $\minv f:f(\Omega)\to \mathcal A_d(\Omega)$ satisfies Lusin's properties $(N)$ and $(N\inv)$: Given $E\subset f(\Omega)$, we have 
\[\Ha^n(\minv f(E))=0 \Longleftrightarrow |E|=0.
\]
Moreover, $\minv f$ is differentiable almost everywhere on $f(\Omega)$ and satisfies the area formula
\begin{align}\label{eq:area-formula}
\int_{f(\Omega)} g\circ\minv f\ \J \minv f\ud y=\int_{\Omega_f}g\ud\Ha^n,
\end{align}
where $\Omega_f:=\minv f(f(\Omega))\subset \mathcal A_d(\Omega)$.
\end{proposition}

Property $(N)$ and the area formula follow from the pseudomonotonicity of $\minv f$ (see Lemma \ref{lem:pseudomonotone} below) by results in \cite{hei00,kar07}, while the inverse Lusin property $(N\inv)$ is a consequence of Proposition \ref{prop:hausd-meas-est}. Recall that a map $F:X\to Y$ is pseudomonotone if there exists a constants $K,r_0>0$ such that 
\[
\diam F(B(x,r))\le K\diam\partial F(B(x,r))
\]
for $x\in X$ and $r\in (0,r_0)$. 

\begin{lemma}\label{lem:pseudomonotone}
Let $f:\Omega\to \R^n$ be a branched cover of degree $d$. Then for every compact $K\subset f(\Omega)$ there exists $r_0>0$ such that
\[
\diam F(B(y,r))\le \sqrt d \diam \partial F(B(y,r))
\]
for all $y\in K$ and $r\in (0,r_0)$.
\end{lemma}
\begin{proof}
Let $r_0$ be the Lebesgue number of the open cover $\{B(y,r_y)\}_{y\in K}$ where, for each $y\in K$, $r_y>0$ is such that 
\[
f\inv B(y,r_y)=\bigcup_{x\in f\inv(y)}U_f(x,r_y)
\]
and each $U_f(x,r_y)$ is a normal domain for $f$ compactly contained in $\Omega$. Now for each $y\in K$ and $r<r_y$ we have $B(y,r)\subset B(y_0,r_{y_0})$ for some $y_0\in K$. Writing
\begin{align*}
f\inv B(y,r)=\bigcup_{x\in f\inv(y)}U_f(x,r),
\end{align*}
we have 
\begin{align*}
\diam F(B(y,r))^2&\le \sum_{x\in f\inv(y)}\iota(f,x)\diam U_f(x,r)^2\le d\max_{x\in f\inv(y)}\diam U_f(x,r)^2\\
&\le d \max_{x\in f\inv(y)}\diam \partial U_f(x,r)^2\le d \diam\partial F(B(y,r))^2.
\end{align*}
See \cite[Chapter I, Lemma 4.7]{ric93} for the fact that $U_f(x,r)$ is a normal domain of $f$ and thus $f(\partial U_f(x,r))=\partial B(y,r)$, implying 
\[
\diam \partial U_f(x,r)\le \diam F(\partial B(y,r))=\diam\partial F(B(y,r)).
\]
\end{proof}

\begin{proof}[Proof of Proposition \ref{prop:area-formula}]
Since $\minv f\in W^{1,n}_{loc}(f(\Omega),\mathcal A_d(\R^n))$ is continuous and pseudomonotone, it has Lusin's property $(N)$ by \cite[Theorem 7.2]{hei00}. Moreover, by Corollary 3.4 in \cite{kar07} $\minv f$ is metrically differentiable almost everywhere, while by Corollary 3.9(1) in \cite{kar07}, $\minv f$ satisfies the area formula \eqref{eq:area-formula}. 

The a.e. metric differentiability of $f$ implies that $\Lip f(x)<\infty$ a.e. $x\in f(\Omega)$, thus by the Stepanov theorem for multivalued maps (see \cite[Corollary 2.8]{men10}) it follows that $f$ is differentiable almost everywhere.

It remains to prove Lusin's inverse property $(N\inv)$. This follows from Proposition \ref{prop:hausd-meas-est} with $s=n$: if $E\subset f(\Omega)$ and $\Ha^n(\minv f(E))=0$, then 
\[
|f\inv(E)|=|(\minv f\circ f)\inv(\minv f(E))|\le d\cdot\Ha^n(\minv f(E))=0,
\]
which implies $|E|=0$ by the Lusin property $(N)$ of $f$.
\end{proof}

\section{Pull-back of differential forms by multi-valued Sobolev maps}\label{sec:pull-back-diff-forms}
\subsection{Differential forms on Almgren space}\label{sec:diff-forms}

Let $\Gamma$ be a finite group acting by isometries on $\R^N$ and let $W\subset \R^N$ be a $\Gamma$-invariant open set. A smooth differential form $\omega\in \Omega^k(W)$ is called $\Gamma$-invariant, if $\gamma^*\omega=\omega$ for all $\gamma\in \Gamma$. $\Gamma$-invariant differential forms can be considered differential forms on $W/ \Gamma$, even when this space is not a manifold. The particular case $\Gamma=S_d$ gives rise to differential forms on Almgren space.

\begin{definition}
Suppose $V\subset \R^n$ is an open set, $d\ge 1$, and $k\in\N$. A (smooth) $k$-form on $\mathcal A_d(V)$ is a (smooth) $S_d$-invariant $k$-form $\omega\in \Omega^k(V^d)$. We denote the space of smooth $k$-forms on $\mathcal A_d(V)$ by $\Omega^k(\mathcal A_d(V))$.
\end{definition}

An element $\omega\in \Omega^k(\mathcal A_d(V))$ is a smooth map $\omega:V^d\to \bigwedge^k(\R^n)^d$ with 
\begin{align}\label{eq:S_d-invariant}
\omega_x(v_1,\ldots,v_k)=\omega_{\sigma(x)}(\sigma v_1,\ldots,\sigma v_k),\quad x\in V^d,\ v_1,\ldots, v_k\in (\R^n)^d.
\end{align}
In particular, for each $\bb{x}\in \mathcal A_d(V)$, the value $\omega_{\bb{x}}:=\omega_{\sigma(x)}\circ \sigma\in \bigwedge^k(\R^n)^d$
is well-defined and independent of $\sigma\in S_d$. We define the comass $\|\omega\|:\mathcal A_d(V)\to \R$ of a $k$-form $\omega\in \Omega^k(\mathcal A_d(V))$ by 
\begin{align*}
\|\omega\|_{\bb x}=\|\omega_{x}\|,\quad \bb x\in \mathcal A_d(V).
\end{align*}
This is well-defined by \eqref{eq:S_d-invariant}.

The next lemma provides a projection map from the space of differential forms to the space of invariant differential forms. Later on we will need Lemma \ref{lem:proj-onto-invariant-forms} for $\Gamma=S_{d_0}\times S_{d_1}\le S_{d_0+d_1}$ as well as $\Gamma=S_d$.

\begin{lemma}\label{lem:proj-onto-invariant-forms}
Let $\Gamma$ be a finite group acting by linear isometries on $\R^N$, and $W\subset \R^N$ a $\Gamma$-invariant open set. For any $\omega\in \Omega^k(W)$, define
\begin{align*}
P_\Gamma\omega=\frac 1{|\Gamma|}\sum_{\gamma\in \Gamma}\gamma^*\omega.
\end{align*}
Then $P_\Gamma:\Omega^k(W)\to \Omega^k(W)$ is linear and continuous with respect to pointwise converge of forms. Moreover, for any $\omega\in \Omega^k(W)$,
\begin{itemize}
    \item[(i)] $P_\Gamma\omega$ is $\Gamma$-invariant,
    \item[(ii)] $P_\Gamma\omega=\omega$ if and only if $\omega$ is $\Gamma$-invariant,
    \item[(iii)] $P_\Gamma(\ud\omega)=\ud(P_\Gamma\omega)$,
    \item[(iv)] $\|P_\Gamma\omega\|_\infty\le \|\omega\|_\infty$.
\end{itemize}
\end{lemma}
\begin{proof}
Linearity and continuity follow from the corresponding properties of the pull-back $\gamma^*\omega$. The identity
\begin{align*}
\gamma_0^*(P_\Gamma\omega)=\frac 1{|\Gamma|}\sum_{\gamma\in\Gamma}\gamma_0^*\gamma^*\omega=\frac 1{|\Gamma|}\sum_{\gamma\in\Gamma}(\gamma\circ\gamma_0)^*\omega=P_\Gamma\omega,\quad \gamma_0\in \Gamma
\end{align*}
proves (i). To see (ii), note that if $\omega$ is $\Gamma$-invariant, then clearly $P_\Gamma\omega=\omega$, while if $P_\Gamma\omega=\omega$, then $\omega$ is $\Gamma$-invariant by (i). 

Finally, (iii) follows from the fact that pull-back and exterior derivative commute:
\[
\ud(P_\Gamma\omega)=\frac 1{|\Gamma|}\sum_{\gamma\in\Gamma}\ud(\gamma^*\omega)=\frac 1{|\Gamma|}\sum_{\gamma\in\Gamma}\gamma^*(\ud\omega)=P_\Gamma(\ud\omega),
\]
and (iv) follows from the estimate
\[
\|P_\Gamma\omega_x\|\le \frac 1{|\Gamma|}\sum_{\gamma\in \Gamma}\|\gamma^*\omega\|_\infty\le \frac 1{|\Gamma|}\sum_{\gamma\in \Gamma}\|\omega\|_\infty\le \|\omega\|_\infty.
\]
\end{proof}
When $\Gamma=S_d$ acts on $(\R^n)^d$ (see Section \ref{sec:almgren-space}), we denote $\bar\omega:=P_{S_d}\omega$, and single out the following corollary.

\begin{corollary}
The map $P=\omega\mapsto \bar\omega:\Omega^k(V^d)\to \Omega^k(\mathcal A_d(V))$ is a linear projection onto $\Omega^k(\mathcal A_d(V))$ satisfying $\ud(P\omega)=P(\ud\omega)$. In particular, smooth $k$-forms on $\mathcal A_d(V)$ form a vector space and are closed under the exterior derivative.
\end{corollary}

\subsubsection*{Canonical $n$-form on $\mathcal A_d(\R^n)$} Let $\alpha$ be a smooth $k$-form on an open set $V\subset \R^n$ and consider the trace $\operatorname{tr}_d(\omega)=\operatorname{tr}(\omega)$, which is a $k$-form on the product $V^d$, defined by
\begin{align}\label{eq:trace}
\operatorname{tr}(\alpha)=\sum_j^dP_j^*\alpha,
\end{align}
where $P_j:(\R^n)^d\to \R^n$ is the projection to the $j^\textrm{th}$ factor. It is straightforward to see from \eqref{eq:trace} that $\operatorname{tr}(\alpha)$ is a smooth $k$-form on $\mathcal A_d(V)$. 
\begin{definition}\label{def:natural-n-form}
We define a natural $n$-form $\omega_n$ on $\mathcal A_d(\R^n)$ by $\omega_n:=\operatorname{tr}(\operatorname{vol}_{\R^n})$, where $\operatorname{vol}_{\R^n}$ is the volume form on $\R^n$.
\end{definition}

\begin{lemma}\label{lem:comass-of-natural-n-form}
We have that $\|\omega_n\|=1$ pointwise everywhere.
\end{lemma}
\begin{proof}
Let $v_l=(v_l^1,\ldots,v_l^d)\in (\R^n)^d$ for $l=1,\ldots,n$ be vectors with $|v_l|^2=\displaystyle\sum_j^d|v_l^j|^2\le 1$. Then, using the Hadamard inequality and the generalized Hölder inequality, we may estimate
\begin{align*}
\omega_n(v_1,\ldots,v_n)&=\sum_j^d v_1^j\wedge\cdots\wedge v_n^j\le \sum_j^d|v_1^j|\cdots|v_n^j|\\
&\le \Big(\sum_j^d|v_1^j|^n\Big)^{1/n}\cdots\Big(\sum_j^d|v_n^j|^n\Big)^{1/n}\le 1.
\end{align*}
Thus $\|\omega_n\|\le 1$ everywhere. On the other hand, choosing $v_l=(e_l,0,\ldots,0)$ yields $\omega_n(v_1,\ldots,v_n)=e_1\wedge\cdots\wedge e_n=1$, implying $\|\omega_n\|\ge 1$. This completes the proof.
\end{proof}

\subsection{Pull-back by multi-valued maps}\label{sec:pull-back} The purpose of this section is to define pull-backs of forms on $\mathcal A_d(\R^n)$ by multi-valued maps, and to prove Theorem \ref{thm:sob-pull-back}. The Sobolev regularity of pull-backs is interesting in its own right, and is also necessary for the proof of Corollary \ref{cor:QR-curve} and other applications. We start by defining the pull-back.

Suppose  $U\subset \R^m$ and $V\subset \R^n$ are open, $h:U\to\mathcal A_d(V)$ is approximately differentiable at $x\in U$, and $h(x)=\bb{y_1,\ldots,y_d}$ and $D_xh=\bb{D_xh_1,\ldots,D_xh_d}$ are labelings such that (i) and (ii) in Theorem \ref{thm:multi-valued-rademacher} hold. By \eqref{eq:S_d-invariant}, the element 
\begin{align*}
\omega_{y_1,\ldots,y_d}\circ(D_xh_1,\ldots,D_xh_d)\in \bigwedge^k\R^m
\end{align*}
is independent of the permutation, i.e.
\[
\omega_{y_{\sigma\inv(1)},\ldots,y_{\sigma\inv(d)}}\circ(D_xh_{\sigma\inv(1)},\ldots,D_xh_{\sigma\inv(d)})=\omega_{y_1,\ldots,y_d}\circ(D_xh_1,\ldots,D_xh_d)
\]
for all $\sigma\in S_d.$ We denote this common element of $\bigwedge^k\R^m$ by $\omega_{h(x)}\circ D_xh$.

\begin{definition}\label{def:pull-back}
Suppose $U\subset \R^m$ and $V\subset \R^n$ are open sets. Let $\omega$ be a smooth $k$-form on $\mathcal A_d(V)$ and $h:U\to \mathcal A_d(V)$ approximately differentiable a.e. in $U$. We define a $k$-form $h^*\omega:U\to \bigwedge^k\R^m$ (defined almost everywhere) by
\[
(h^*\omega)_x=\omega_{h(x)}\circ D_xh.
\]
whenever $h$ is approximately differentiable at $x\in U$.
\end{definition}

We now state the main result in the proof of Theorem \ref{thm:sob-pull-back}, which states that pull-backs by multi-valued Lipschitz maps are flat forms.

\begin{theorem}\label{thm:lip-pullback-flat-form}
Suppose $U\subset\R^m$ is an open set, $f:U\to\mathcal A_d(\R^n)$ is Lipschitz and $\omega\in \Omega^k(\mathcal A_d(\R^n))$. Then $\ud(f^*\omega)=f^*(\ud\omega)$ weakly. In particular, $f^*\omega$ is a flat form on $U$ in the sense of Whitney. 
\end{theorem}

Theorem \ref{thm:sob-pull-back} follows from Theorem \ref{thm:lip-pullback-flat-form} by a Lusin-type approximation, cf. Proposition \ref{prop:lusin-approx}. The proof of Theorem \ref{thm:lip-pullback-flat-form} is more involved; it uses local decompositions of multi-valued functions, the density of tensor products of differential forms, and careful approximation arguments.

We begin with a simple estimate of the comass of the pull-back.

\begin{lemma}\label{lem:pullback-est}
Suppose $U\subset \R^m$ and $V\subset \R^n$ are open sets, $\omega$ is a smooth $k$-form on $\mathcal A_d(V)$ and $h:U\to \mathcal A_d(V)$ a.e. approximately differentiable. Then $h^*\omega$ is a measurable $k$-form on $U$ and 
\begin{align*}
\|h^*\omega\|_x\le \|D_xh\|_x^k\|\omega_{h(x)}\|\quad\textrm{a.e. }x\in U.
\end{align*}
\end{lemma}
\begin{remark}
In particular $h^*\omega\in L^{p/k}_{loc}(U)$ if $h\in W^{1,p}_{loc}(U,\mathcal A_d(V))$.
\end{remark}
\begin{proof}
The approximate differentials $Dh_1,\ldots,Dh_d$ are measurable. For all $x\in U$ where $h$ is approximately differentiable, we have
\begin{align*}
\|h^*\omega\|_x&=\max\{\omega_{h(x)}(Dh(v_1),\ldots,Dh (v_d)):\ |v_1|,\ldots,|v_d|\le 1\}\\
&\le \|D_xh\|_{op}^k\|\omega_{h(x)}\|,
\end{align*}
completing the proof.
\end{proof}

For the next lemma, recall the definition of the tensor product of differential forms. For (open) manifolds $U_0,U_1$, the projections $P_i:U_0\times U_1\to U_i$ induce linear maps $P_i^*:\Omega^l(U_i)\to \Omega^l(U_0\times U_1)$ which can be combined to obtain a bi-linear map
\[
P^*:\Omega^l(U_0)\times \Omega^r(U_1)\to \Omega^k(U_0\times U_1),\quad (\omega_0,\omega_1)\mapsto P_0^*\omega_0\wedge P_1^*\omega_1
\]
whenever $l+r=k$. The associated linear map $P^*:\Omega^l(U_0)\otimes \Omega^r(U_1)\to \Omega^k(U_0\times U_1)$ is injective, and 
\begin{align*}
P^*:\bigcup_{r+l=k}\Omega^l(U_0)\otimes \Omega^r(U_1)\to \Omega^k(U_0\times U_1)
\end{align*}
has dense range in the natural Frechet topology on $\Omega^k(U_0\times U_1)$. We denote $\omega_0\otimes\omega_1:=P^*(\omega_0,\omega_1)$.

Note that, if $\omega_i\in \Omega^{k_i}(\mathcal A_{d_i}(\R^n))$ for $i=0,1$, and we denote $k=k_0+k_1$, $d=d_0+d_1$, the tensor product $\omega_0\otimes \omega_1\in \Omega^k((\R^n)^d)$ is not $S_d$-invariant but merely $S_{d_0}\times S_{d_1}$-invariant.

\begin{lemma}
Let $\omega_i\in \Omega^{k_i}(V^{d_i})$, $i=0,1$, and $k=k_0+k_1$, $d=d_0+d_1$. Then 
\begin{align}\label{eq:otimes}
P_{S_{d_0}\times S_{d_1}}(\omega_0\otimes \omega_1)=(P_{S_{d_0}}\omega_0)\otimes(P_{S_{d_1}}\omega_1).
\end{align}
In particular, if $\omega_i$ is a $k_i$-form on $\mathcal A_{d_i}(V)$ for each $i=0,1$, then $\omega_0\otimes \omega_1$ is $S_{d_0}\times S_{d_1}$-invariant.
\end{lemma}
\begin{proof}
The last claim follows directly from \eqref{eq:otimes} and Lemma \ref{lem:proj-onto-invariant-forms}. To show \eqref{eq:otimes}, let $(\sigma_0,\sigma_1)\in S_{d_0}\times S_{d_1}$. For any $v\in (\R^n)^d$, $(\sigma_0,\sigma_1)v=:v^{(\sigma_0,\sigma_1)}$ satisfies $P_0(v^{(\sigma_0,\sigma_1)})=\sigma_0(P_0v)$ and $P_1(v^{(\sigma_0,\sigma_1)})=\sigma_1(P_1v)$. Thus
\begin{align*}
(\sigma_0,\sigma_1)^*(P_0^*\omega_0&\wedge P_1^*\omega_1)(v_1,\ldots,v_k)=(P_0^*\omega_0\wedge P_1^*\omega_1)(v_1^{(\sigma_0,\sigma_1)},\ldots,v_k^{(\sigma_0,\sigma_1)})\\
=\frac 1{k!}\sum_{\theta\in S_k}\sgn(\theta)&(P_0^*\omega_0)(v_{\theta(1)}^{(\sigma_0,\sigma_1)},\ldots,v_{\theta(k_0)}^{(\sigma_0,\sigma_1)})(P_1^*\omega_1)(v_{\theta(k_0+1)}^{(\sigma_0,\sigma_1)},\ldots,v_{\theta(k)}^{(\sigma_0,\sigma_1)})\\
=\frac 1{k!}\sum_{\theta\in S_k}\sgn(\theta)&\omega_0(\sigma_0(P_0v_{\theta(1)}),\ldots , \sigma_0(P_0v_{\theta(k_0)}))\\
&\times\omega_1(\sigma_1(P_1v_{\theta(k_0+1)}),\ldots,\sigma_1(P_1v_{\theta(k)}))\\
=P_0^*(\sigma_0^*\omega_0)\wedge P_1&^*(\sigma_1^*\omega_1)(v_1,\ldots,v_k)
\end{align*}
for all $v_1,\ldots,v_k\in (\R^n)^d$. Now we can calculate
\begin{align*}
P_{S_{d_0}\times S_{d_1}}(\omega_0\otimes \omega_1)=&P_{S_{d_0}\times S_{d_1}}(P_0^*\omega_0\wedge P_1^*\omega_1)\\
=&\frac{1}{|S_{d_0}||S_{d_1}|}\sum_{\sigma_0\in S_{d_0}}\sum_{\sigma_1\in S_{d_1}}(\sigma_0,\sigma_1)^*(P_0^*\omega_0\wedge P_1^*\omega_1)\\
=&\frac{1}{|S_{d_0}||S_{d_1}|}\sum_{\sigma_0\in S_{d_0}}\sum_{\sigma_1\in S_{d_1}}P_0^*(\sigma_0^*\omega_0)\wedge P_1^*(\sigma_1^*\omega_1)\\
=&P_0^*(P_{S_{d_0}}\omega_0)\wedge P_1^*(P_{S_{d_1}}\omega_1)=(P_{S_{d_0}}\omega_0)\otimes (P_{S_{d_1}}\omega_1).
\end{align*}
\end{proof}

In the next lemma we show that $S_{d_0}\times S_{d_1}$-invariant forms have a well-defined pull back by multivalued functions $f=\bb{f_0,f_1}$, where $f_i:U\to \mathcal A_{d_i}(\R^n)$ is Lipschitz for each $i=0,1$. We emphasize that $S_{d_0}\times S_{d_1}$-invariant forms \emph{cannot} be pulled back by general multivalued Lipchitz maps $f:U\to \mathcal A_d(\R^n)$, that is, we are using the additional information of having a decomposition in Lemma \ref{lem:pull-back-of-decomp}.

\begin{lemma}\label{lem:pull-back-of-decomp}
Suppose $f_i:U\to \mathcal A_{d_i}(V)$ is Lipschitz ($i=0,1$), $d=d_0+d_1$, and $\omega$ is a $S_{d_0}\times S_{d_1}$-invariant form on $V^d$. Then the pull-back $\bb{f_0,f_1}^*\omega$ is a well-defined $k$-form on $U$. Moreover, 
\begin{align*}
\bb{f_0,f_1}^*(\omega_0\otimes\omega_1)=f_0^*\omega_0\wedge f_1^*\omega_1
\end{align*}
whenever $\omega_i\in \Omega^{k_i}(\mathcal A_{d_i}(V))$ ($i=0,1$) and $k=k_0+k_1$.
\end{lemma}
\begin{proof}
Let $x\in U$ be such that $f_0,f_1$ are differentiable at $x$, and let $f_i^1(x),\ldots, f_i^{d_i}(x)$ and $D_xf_i^1,\ldots,D_xf_i^{d_i}$ be as in Theorem \ref{thm:multi-valued-rademacher} (i) and (ii). (In particular, note that $f=\bb{f_0,f_1}$ is differentiable at $x$ and $Df=\bb{D_xf_0^1,\ldots,D_xf_0^{d_0},D_xf_1^1,\ldots,D_xf_1^{d_1}}$.) Given $(\sigma_0,\sigma_1)\in S_{d_0}\times S_{d_1}$, denote
\[
D_xf^{(\sigma_0,\sigma_1)}=(D_xf_0^{\sigma_0\inv(1)},\ldots D_xf_0^{\sigma_0\inv(d_0)},D_xf_1^{\sigma_1\inv(1)},\ldots,D_xf_1^{\sigma_1\inv(d_1)})
\]
and similarly 
\[
(f_0^{\sigma_0}(x),f_1^{\sigma_1}(x))=(f_0^{\sigma_0\inv(1)}(x),\ldots,f_0^{\sigma_0\inv(d_0)},f_1^{\sigma_1\inv(1)},\ldots,f_1^{\sigma_1\inv(d_1)}).
\]
We define
\begin{align*}
    (f^*\omega)_x(v_1,\ldots,v_k)=\omega_{(f_0^{\sigma_0}(x),f_1^{\sigma_1}(x))}(D_xf^{(\sigma_0,\sigma_1)}v_1,\ldots,D_xf^{(\sigma_0,\sigma_1)}v_k)
\end{align*}
for $v_1,\ldots,v_k\in T_xU=\R^m$. By the $S_{d_0}\times S_{d_1}$-invariance of $\omega$, this expression is independent of $(\sigma_0,\sigma_1)$ and gives a well-defined element in $\bigwedge^k\R^m$, which we denote by $(f^*\omega)_x=\omega_{(f_0(x),f_1(x))}\circ(D_xf_0,D_xf_1)$.

In particular, if $\omega_i\in \Omega^{k_i}(\mathcal A_{d_i}(V))$, $i=0,1$, then $\omega_0\otimes \omega_1$ is $S_{d_0}\times S_{d_1}$-invariant. We may calculate
\begin{align*}
&f^*(\omega_0\otimes\omega_1)(v_1,\ldots,v_k)=(\omega_0\otimes\omega_1)((D_xf_0,D_xf_1) v_1,\ldots,(D_xf_0,D_xf_1) v_k)\\
=&\frac 1{k!}\sum_{\theta\in S_k}\sgn(\theta)\omega_0(Df_0v_{\theta(1)},\ldots,Df_0v_{\theta(k_0)})\omega_1(Df_1v_{\theta(k_0+1)},\ldots,Df_1v_{\theta(k)})\\
=&(f_0^*\omega_0)\wedge(f_1^*\omega_1)(v_1,\ldots,v_k)
\end{align*}
for every $v_1,\ldots,v_k\in \R^m$. This completes the proof.
\end{proof}

\begin{proposition}\label{prop:flat-form}
Let $d_0,d_1\ge 1$ be natural numbers, $d:=d_0+d_1$, and let $U\subset \R^m$ be an open set. If $f_i:U\to \mathcal A_{d_i}(\R^n)$ are Lipschitz functions such that $\ud(f_i^*\omega)=f_i^*(\ud\omega)$ weakly for all $l$-forms $\omega\in \Omega^l(\mathcal A_{d_i}(\R^n))$, $l\le k$ ($i=0,1$), then $\ud(\bb{f_0,f_1}^*\omega)=\bb{f_0,f_1}^*(\ud\omega)$ weakly for all $k$-forms $\omega\in \Omega^k(\mathcal A_d(\R^n))$.
\end{proposition}
\begin{proof}
Denote $f:=\bb{f_0,f_1}$, and let $\omega\in \Omega^k(\mathcal A_d(\R^n))$. Let 
\[
(\omega_j)\subset \bigcup_{r+l=k}\Omega^l((\R^n)^{d_0})\otimes \Omega^r((\R^n)^{d_1})
\]
be a sequence such that $(\omega_j)_x\to \omega_x$ and $(\ud\omega_j)_x\to (\ud\omega)_x$ for all $x\in (\R^n)^d$ and $\sup_j\max\{\|\omega_j\|_\infty,\|\ud\omega_j\|_\infty\}<\infty$.  Then
\begin{align*}
\tilde\omega_j&:=P_{S_{d_0}\times S_{d_1}}\omega_j\to P_{S_{d_0}\times S_{d_1}}\omega=\omega,\\ 
\ud\tilde\omega_j&=P_{S_{d_0}\times S_{d_1}}(\ud\omega_j)\to P_{S_{d_0}\times S_{d_1}}(\ud\omega)=\ud\omega
\end{align*}
and $\sup_j\max\{\|\tilde \omega_j\|_\infty,\|\ud\tilde \omega_j\|_\infty\}<\infty$, cf. Lemma \ref{lem:proj-onto-invariant-forms}. By Lemma \ref{lem:pull-back-of-decomp} each $\tilde\omega_j$ is a sum of tensor products $\omega_0\otimes\omega_1$ where $\omega_i$ is $S_{d_i}$-invariant and $f^*(\omega_0\otimes\omega_1)=(f_0^*\omega_0)\wedge(f^*_1\omega_1)$. Together with the assumption that $f_i^*(\ud\omega_i)=\ud(f^*_i\omega)$ weakly for $S_{d_i}$-invariant forms $\omega$, this yields 
\begin{align*}
f^*(\ud(\omega_0\otimes\omega_1))&=f^*(\ud\omega_0\otimes \omega_1+(-1)^{\deg\omega_0}\omega_0\otimes\ud\omega_1)\\
&=f^*_0(\ud\omega_0)\wedge(f^*_1\omega_1)+(-1)^{\deg\omega_0}(f_0^*\omega_0)\wedge (f_1^*(\ud\omega_1))\\
&=\ud(f_0^*\omega_0)\wedge(f_1^*\omega_1)+(-1)^{\deg\omega_0}(f_0^*\omega_0)\wedge\ud(f_1^*\omega_1)\\
&=\ud(f_0^*\omega_0\wedge f_1^*\omega_1)=\ud f^*(\omega_0\otimes\omega_1)
\end{align*}
weakly. Consequently $f^*(\ud\tilde\omega_j)=\ud(f^*\tilde\omega_j)$ weakly. Now the pointwise convergence $f^*\tilde\omega_j\to f^*\omega$, $f^*(\ud\tilde\omega_j)\to f^*(\ud\omega)$ together with the dominated convergence theorem yield
\begin{align*}
\int_U \alpha\wedge f^*(\ud\omega)&=\lim_{j\to\infty}\int_U \alpha\wedge f^*(\ud\tilde\omega_j)=(-1)^{k+1}\lim_{j\to\infty}\int_U\ud\alpha\wedge f^*\tilde\omega_j\\
&=(-1)^{k+1}\int_U\ud\alpha\wedge f^*\omega
\end{align*}
for all $\alpha\in \Omega_c^{m-k-1}(U)$. This proves the claim. 
\end{proof}

The next proposition gives uniform control of the Lipschitz constants of the approximations.

\begin{proposition}\label{prop:multi-valued-approx}
Let $f:U\to \mathcal A_d(\R^n)$ be $L$-Lipschitz. Given $\varepsilon>0$, denote
\begin{align*}
F_\varepsilon=\{x\in U: d_\mathcal A(f(x),d\bb{b(f(x))}<\varepsilon\},\quad \eta_\varepsilon=\Big(1-\frac{\dist(F_\varepsilon,\cdot)}{\varepsilon}\Big)_+
\end{align*}
Define $f_\varepsilon:U\to \mathcal A_d(\R^n)$ by
\begin{align*}
f_\varepsilon=\bb{(1-\eta_\varepsilon)f_1+\eta_\varepsilon b(f),\ldots, (1-\eta_\varepsilon)f_d+\eta_\varepsilon b(f)},
\end{align*}
where $f=\bb{f_1,\ldots,f_d}$. Then $f_\varepsilon\to f$ locally uniformly and $\LIP(f_\varepsilon)\le (3+2d)L$.
\end{proposition}
\begin{proof}
Denote $g_j=b(f)-f_j$ and note that $f_\varepsilon=\bb{f_1+\eta_\varepsilon g_1,\ldots,f_d+\eta_\varepsilon g_d}$. Observe moreover that (a) $\eta_\varepsilon$ is $\varepsilon\inv$-Lipschitz, (b) $\eta_\varepsilon|_{F_\varepsilon}=1$, and (c) $\spt(\eta_\varepsilon)\subset B(F_\varepsilon,\varepsilon)\subset \overline F_{(1+L)\varepsilon}$. We may assume $L\ge 1$ so that (c) implies $\spt(\eta_\varepsilon)\subset F_{2L\varepsilon}$.

Fix $x,y\in U$. For any $\sigma \in S_d$ we have 
\begin{align*}
&\Big(\sum_j^d|f_j(x)-f_{\sigma(j)}(y)+\eta_\varepsilon(x)g_j(x)-\eta_\varepsilon(y)g_{\sigma(j)}(y)|^2\Big)^{1/2}\\
\le &\Big(\sum_j^d|f_j(x)-f_{\sigma(j)}(y)|^2\Big)^{1/2}+\Big(\sum_j^d|\eta_\varepsilon(x)g_j(x)-\eta_\varepsilon(y)g_{\sigma(j)}(y)|^2\Big)^{1/2}
\end{align*}
If $x,y\notin F_{2L\varepsilon}$, then $|\eta_\varepsilon(x)g_j(x)-\eta_\varepsilon(y)g_{\sigma(j)}(y)|=0$, while if $x\in F_{2L\varepsilon}$ or $y\in F_{2L\varepsilon}$, then $\min\{|g_j(x)|,g_{\sigma(j)}(y)\}\le 2L\varepsilon$ and $\max\{\eta_\varepsilon(x),\eta_\varepsilon(y)\}\le 1$. Consequently
\begin{align*}
|\eta_\varepsilon(x)g_j(x)-\eta_\varepsilon(y)g_{\sigma(j)}(y)|\le &|\eta_\varepsilon(x)-\eta_\varepsilon(y)|2L\varepsilon+|g_j(x)-g_{\sigma(j)}(y)|\\
\le &2Ld(x,y)+|f_j(x)-f_{\sigma(j)}(y)|+|b(f(x))-b(f(y))|\\
\le & 2L|x-y|+\frac 1{\sqrt d}d_\mathcal A(f(x),f(y))+|f_j(x)-f_{\sigma(j)}(y)|.
\end{align*}
Thus we obtain the estimate
\begin{align*}
\Big(\sum_j^d|(f_\varepsilon)_j(x)-(f_\varepsilon)_{\sigma(j)}(y)|^2\Big)^{1/2}&\le 2\Big(\sum_j^d|f_j(x)-f_{\sigma(j)}(y)|^2\Big)^{1/2}\\
+&2Ld\cdot |x-y|+d_\mathcal A(f(x),f(y)).
\end{align*}
Taking infimum over $\sigma\in S_d$ yields
\begin{align*}
    d_\mathcal A(f_\varepsilon(x),f_\varepsilon(y))&\le 3d_\mathcal A(f(x),f(y))+2Ld|x-y|\\
    &\le (3+2d)L|x-y|.
\end{align*}
It remains to show that $f_\varepsilon\to f$ locally uniformly in $U$. Since $\LIP(f_\varepsilon)$ is bounded uniformly in $\varepsilon$, it suffices to prove pointwise convergence. To this end observe that 
\begin{align*}
d_\mathcal A(f_\varepsilon(x),f(x))&\le \eta_\varepsilon(x)\Big(\sum_j^d|f_j(x)-b(f(x))|^2\Big)^{1/2}\\
&=\eta_\varepsilon(x)d_\mathcal{A}(f(x),d\bb{b(f(x)})\le 2L\varepsilon.
\end{align*}
Indeed, if $x\notin F_{2L\varepsilon}$ then $\eta_\varepsilon(x)=0$ and otherwise $d_\mathcal{A}(f(x),b(f(x))\le 2L\varepsilon$. This completes the proof.
\end{proof}

\begin{proof}[Proof of Theorem \ref{thm:lip-pullback-flat-form}]
We prove the claim by induction on $d$. For $d=1$ the result is classical. Suppose $d>1$ and that $g^*(\ud\omega)=\ud(g^*\omega)$ weakly for any locally Lipschitz map $g:U'\to \mathcal A_l(\R^n)$ from an open set $U'\subset \R^m$ with $l<d$. We shall decompose the original map $f$ into sums of locally simpler maps and utilize Proposition \ref{prop:flat-form}.

Denote $F:=f\inv(\diag\mathcal A_d(\R^n))$ and $L=\LIP(f)$. Observe that $F$ is a closed set in $U$, and that $f=d\bb{b(f)}$ on $F$. 

Given $\varepsilon>0$, set 
\[
F_\varepsilon=\{x: d_{\mathcal A}(f(x),b(f(x))< \varepsilon \}\quad\textrm{and}\quad\eta_\varepsilon(y)=(1-\varepsilon\inv\dist(F_\varepsilon,y))_+
\]
as in Proposition \ref{prop:multi-valued-approx}. Recall that $F_\varepsilon$ is open and $\displaystyle \bigcap_{\varepsilon>0}F_\varepsilon=F$, and that (a) $\eta_\varepsilon$ is $\varepsilon\inv$-Lipschitz, (b) $\eta_\varepsilon|_{F_\varepsilon}=1$, and (c) $\spt(\eta_\varepsilon)\subset B(F_\varepsilon,\varepsilon)\subset \overline F_{(1+L)\varepsilon}$. The map $f_\varepsilon:U\to \mathcal A_d(\R^m)$ 
\begin{align*}
f_\varepsilon=\bb{(1-\eta_\varepsilon)f_1+\eta_\varepsilon b(f),\ldots,(1-\eta_\varepsilon)f_d+\eta_\varepsilon b(f)}
\end{align*}
of Proposition \ref{prop:multi-valued-approx} satisfies the following.
\begin{itemize}
\item[(0)] $f_\varepsilon$ is $(3+2d)L$-Lipschitz;
    \item[(1)] $f_\varepsilon|_{F}=f|_F$ and $f_\varepsilon|_{U\setminus F_{2L\varepsilon}}=f|_{U\setminus F_{2L\varepsilon}}$;
    \item[(2)] $f_\varepsilon=d\bb{b(f)}$ on $F_\varepsilon$;
    \item[(3)] For each $x\in U$ there exists a ball $B_x=B(x,r_x)\subset U$ such that has a decomposition $f_\varepsilon|_{B_x}=\bb{f_{\varepsilon,x}^0,f_{\varepsilon,x}^1}$.
\end{itemize}
Indeed, (0) follows from Proposition \ref{prop:multi-valued-approx}, and (1)-(2) are immediate consequences of the definition. We explain (3): given $x\in U$, if $x\in F_\varepsilon$, then the claim follows from (2). If $x\in U\setminus F_\varepsilon$ then in particular $x\in U\setminus F$ so that $|f_i(x)-f_j(x)|>0$ for some $i,j\in \{1,\ldots,d\}$. Now for these indices 
\[
|(f_\varepsilon)_i(x)-(f_\varepsilon)_i(x)|=(1-\eta_\varepsilon(x))|f_i(x)-f_j(x)|>0.
\]
By \cite[Proposition 1.6]{del11} there exists a ball $B_x=B(x,r_x)\subset U$ and Lipschitz functions $f_{\varepsilon,x}^i:B_x\to \mathcal A_{d_i}(\R^n)$ ($i=0,1$) with $d_i<d$, $d_0+d_1=d$, $\LIP(f_x^i)\le CL$ and $f_\varepsilon=\bb{f_{\varepsilon,x}^0,f_{\varepsilon,x}^1}$ on $B_x$.

Fix a compact set $K\subset U$, and let $\{B_{x_1},\ldots,B_{x_M}\}$ be a finite cover of $\{B_x\}_{x\in K}$ where $B_x$ is as in (3) for each $x$. Let $\{\varphi_1,\ldots\varphi_M\}$ be a partition of unity subordinate to $\{B_{x_1},\ldots,B_{x_M}\}$. We may express $f_\varepsilon^*\omega$ and $f_\varepsilon^*(\ud\omega)$ as
\[
f_\varepsilon^*\omega=\sum_l^M\varphi_l\bb{f_{\varepsilon,x_l}^0,f_{\varepsilon,x_l}^1}^*\omega,\quad f_\varepsilon^*(\ud\omega)=\sum_l^M\varphi_l\bb{f_{\varepsilon,x_l}^0,f_{\varepsilon,x_l}^1}^*(\ud\omega)\quad\textrm{on }K.
\]
Thus for any $\alpha\in \Omega^{m-k-1}(U)$ supported in $K$ Proposition \ref{prop:flat-form} yields
\begin{align*}
\int \ud\alpha\wedge f_\varepsilon^*\omega=&\sum_l^M\int\ud\alpha\wedge(\varphi_l\bb{f_{\varepsilon,x_l}^0,f_{\varepsilon,x_l}^1}^*\omega)=\sum_l^M\int\varphi_l\ud\alpha\wedge \bb{f_{\varepsilon,x_l}^0,f_{\varepsilon,x_l}^1}^*\omega \\
=& \sum_l^M\int\ud(\varphi_l\alpha)\wedge \bb{f_{\varepsilon,x_l}^0,f_{\varepsilon,x_l}^1}^*\omega \\
=&(-1)^{k+1}\sum_l^M\int\varphi_l\alpha\wedge \ud(\bb{f_{\varepsilon,x_l}^0,f_{\varepsilon,x_l}^1}^*\omega)\\
=& (-1)^{k+1}\sum_l^M\int\varphi_l\alpha \wedge\bb{f_{\varepsilon,x_l}^0,f_{\varepsilon,x_l}^1}^*(\ud\omega)=(-1)^{k+1}\int\alpha\wedge f_\varepsilon^*\omega.
\end{align*}
The equality in the second line above follows since $\sum_l^M\varphi_l=1$ on $K$. Since $K$ is arbitrary we obtain 
\begin{align}\label{eq:weak-ext-approx}
\int_U\ud\alpha\wedge f_\varepsilon^*\omega=(-1)^{k+1}\int_U\alpha \wedge f_\varepsilon^*(\ud\omega)
\end{align}
for all $\alpha\in \Omega_c^{m-k-1}(U)$.

But by (2) we have that $f_\varepsilon^*\gamma=f^*\gamma$ almost everywhere outside $F_{2L\varepsilon}\setminus F$ for any differential form $\gamma$ on $\mathcal A_d(\R^n)$. Together with (0) this yields
\begin{align*}
\Big|\int_U\beta\wedge(f_\varepsilon^*\gamma-f^*\gamma) \Big|\le C\|\beta\|_\infty L^{\deg(\gamma)}\|\gamma\|_\infty|F_{2L\varepsilon}\setminus F|\stackrel{\varepsilon\to 0}{\longrightarrow} 0.
\end{align*}
for any $\beta\in \Omega_c^{m-\deg(\gamma)}(U)$. Consequently
\begin{align*}
\int_U\ud\alpha\wedge f^*\omega=\lim_{\varepsilon\to 0}\int_U\ud\alpha\wedge f_\varepsilon^*\omega\stackrel{\eqref{eq:weak-ext-approx}}{=}(-1)^{k+1}\lim_{\varepsilon\to 0}\int_U\alpha\wedge f_\varepsilon^*(\ud\omega)=\int_U\alpha\wedge f^*(\ud\omega)
\end{align*}
for all $\alpha\in \Omega_c^{m-k-1}(U)$. This completes the proof.
\end{proof}

\begin{proof}[Proof of Theorem \ref{thm:sob-pull-back}]
Note that under the assumptions in the claim, we have $f^*\omega\in L^{p/k}_{loc}(U,\bigwedge^k\R^m)$ and $f^*(\ud\omega)\in L^\frac{p}{k+1}_{loc}(U,\bigwedge^{k+1}\R^m)$ by Lemma \ref{lem:pullback-est}. By considering an exhaustion of $U$ by relatively compact domains, we may assume that $U$ is bounded and that $f\in W^{1,p}(U,\mathcal A_d(\R^n))$. In particular $f\in W^{1,k}(U,\mathcal A_d(\R^n))$.

By Proposition \ref{prop:lusin-approx}, for every $\lambda>0$ there exists an $\lambda$-Lipschitz map $f_\lambda:U\to\mathcal A_d(\R^n)$ with
\begin{align*}
\lambda^p|\{f\ne f_\lambda\}|\le C\int_{\{f\ne f_\lambda\}}(|Df|^k+|f|^k).
\end{align*}
We have that $Df=Df_\lambda$ almost everywhere outside $\{f\ne f_\lambda\}$, cf. Proposition \ref{prop:lusin-approx}. Consequently 
\begin{align*}
\Big|\int\beta\wedge(f_\lambda^*\omega-f^*\omega)\Big|&\le C\|\beta\|_\infty \int_{\{f\ne f_\lambda\}}(\|f_\lambda^*\omega\|+\|f^*\omega\|) \\
&\le C\|\beta\|_\infty\|\omega\|_\infty\int_{\{f\ne f_\lambda\}}(\lambda^{k}+|Df|^{k})\ud x\\
&\le C\|\beta\|_\infty\|\omega\|_\infty \int_{\{f\ne f_\lambda\}}(|Df|^k+|f|^k)\stackrel{\lambda\to \infty}{\longrightarrow} 0.
\end{align*}
for all $\beta\in \Omega_c^{m-k}(U)$. Similarly (either since $k+1\le p$ or since $\ud\omega=0$) we obtain
\begin{align*}
\Big|\int_U\alpha\wedge(f_\lambda^*(\ud\omega)-f^*(\ud\omega))\Big|\stackrel{\lambda\to \infty}{\longrightarrow} 0.
\end{align*}
for all $\alpha\in \Omega_c^{m-k-1}(U)$. These identities together with Theorem \ref{thm:lip-pullback-flat-form} imply
\begin{align*}
\int_U\ud\alpha\wedge f^*\omega&=\lim_{\lambda\to\infty}\int_U\ud\alpha\wedge f_\lambda^*\omega=(-1)^{k+1}\lim_{\lambda\to\infty}\int_U\alpha\wedge f_\lambda^*(\ud\omega)\\
&=(-1)^{k+1}\int_U\alpha\wedge f^*(\ud\omega)
\end{align*}
for all $\alpha\in \Omega_c^{m-k-1}(U)$. 
\end{proof}

\section{Multivalued inverse of QR map as a QR curve}\label{sec:QR-curve}
Throughout this section, $\omega=\operatorname{tr}(\operatorname{vol}_{\R^n})$ denotes the canonical $n$-form on $\mathcal{A}_d(\R^n)$.

\begin{lemma}\label{lem:star-vs-jacobian}
Let $f\in W^{1,n}_{loc}(\Omega,\mathcal A_d(\R^n))$, where $\Omega$ is an open set. Then
\begin{align}\label{eq:star-vs-jacobian1}
\star f^*\omega=\sum_j^d \det(Df_j)\quad \textrm{ a.e. on }\Omega.
\end{align}
where $Df=\bb{Df_1,\ldots,Df_d}$ is the weak differential of $f$. In particular, we have the inequality
\begin{align}\label{eq:star-vs-jacobian2}
    |\star f^*\omega|\le \J f\quad\textrm{a.e. on } U.
\end{align}
\end{lemma}

\begin{proof}
Let $Df=\bb{Df_1(x),\ldots,Df_d(x)}$ be the differential of $f$ at $x$. By Definitions \ref{def:natural-n-form} and \ref{def:pull-back} we have that 
\begin{align*}
(f^*\omega)_x(v_1,\ldots,v_n)&=\omega_{f(x)}(Df(x)v_1,\ldots,Df(x)v_n)\\
&=\sum_j^dDf_j(x)v_1\wedge\cdots\wedge Df_j(x)v_n\\
&=\sum_j^d\det(Df_j(x))v_1\wedge\cdots\wedge v_n,\quad v_1,\ldots,v_n\in T_xU,
\end{align*}
which implies \eqref{eq:star-vs-jacobian1}. The estimate \eqref{eq:star-vs-jacobian2} follows from \eqref{eq:star-vs-jacobian1} together with Lemma \ref{lem:metric-jacobian} and Minkowski's determinant inequality:
\begin{align*}
|\star f^*\omega|&\le \sum_j^d|\det(Df_j)|=\sum_j^d[\det(Df_j^TDf_j)]^{1/2}\\
&\le \Big(\sum_j^d[\det(Df_j^TDf_j)]^{1/n}\Big)^{n/2}\le \det\Big(\sum_j^d Df_j^TDf_j\Big)^{1/n\ \cdot\ n/2}= \J f.
\end{align*}
\end{proof}

\begin{proof}[Proof of Theorem \ref{thm:QR-curve}]
The fact that $\minv f\in W^{1,n}_{loc}(f(\Omega),\mathcal A_d(\R^n)$ follows from Theorem \ref{thm:wug} and Corollary \ref{cor:sob-reg}. For every $y\in f(\Omega)\setminus f(B_f)$ we have $\#f\inv(y)=d$ and we can find $r>0$ such that $f\inv:B(y,r)\to U_f(x,r)$ is a homeomorphism for each $x\in f\inv(y)$. Choose a numbering $f\inv(y)=\{x_1,\ldots,x_d\}$ and denote $U_j=U_f(x_j,r)$ and $g_j=f\inv:B(y,r)\to U_j$. Then
\[
\minv f|_{B(y,r)}=\bb{g_1,\ldots,g_d},
\]
and each $g_j$ is a $K_I(f)$-quasiconformal map between domains in $\R^n$, in particular 
\[
\|Dg_j\|^n\le K_IJ_{g_j}=K_I\star g_j^*(\ud x).
\]
Observe that $\star \minv f^*\omega=\sum_j^d J_{g_j}$ and $|D\minv f|^2\le \sum_j^d\|Dg_j\|^2$ a.e. on $B(y,r)$. Thus
\begin{align*}
|D\minv f|^n&\le \Big(\sum_j^d\|Dg_j\|^2\Big)^{n/2}\le d^{n/2-1}\sum_j^d\|Dg_j\|^{n}\\
&\le K_Id^{n/2-1}\sum_j^d J_{g_j}=d^{n/2-1}K_I\star \minv f^*\omega
\end{align*}
a.e. in $B(y,r)$. Since $|f(B_f)|=0$ by \cite[Chapter II, Proposition 4.14]{ric93}, the inequality holds a.e. on $f(\Omega)$ as claimed.
\end{proof}

\begin{remark}\label{rmk:jacobian-vs-star}
An inspection of the proof yields the following inequality:
\begin{align}\label{eq:jacobian-vs-star}
\J \minv f\le K_I\star \minv f^*\omega\quad\textrm{ a.e. on }f(\Omega).
\end{align}
This a reverse version (up to the multiplicative constant $K_I$) of the inequality in Lemma \ref{lem:star-vs-jacobian}. Unlike Lemma \ref{lem:star-vs-jacobian}, \eqref{eq:jacobian-vs-star} does not hold for general Sobolev maps.
\end{remark}

The following higher integrability result for $\minv f$ implies part of Corollary \ref{cor:QR-curve}. 
\begin{theorem}\label{thm:higher-integrability}
If $f:\Omega\to \R^n$ is a non-constant quasiregular map of degree $d$, then there exists $p>n$ such that $\minv f\in W_{loc}^{1,p}(f(\Omega),\mathcal A_d(\Omega))$.
\end{theorem}
\begin{proof}
The proof is identical to Lemma 6.1, Lemma 6.2 and Proposition 6.3 in \cite{pan-onn21}. The only place a minor argument is needed is at the beginning of Lemma 6.1, where it is stated that $\omega=\ud \tau$ for a suitable $(n-1)$-form $\tau$. We need to ensure that $\tau$ can be chosen $S_d$-invariant. But this follows from Lemma \ref{lem:proj-onto-invariant-forms}: $\omega=P_{S_d}\omega=\ud(P_{S_d}\tau)$, and $\tau'=P_{S_d}\tau$ is $S_d$-invariant. The rest of the proof(s) proceed exactly as in \cite{pan-onn21}. Note that $\ud(f^*\tau')=f^*\omega$ (a fact needed in the proof in \cite{pan-onn21}) holds by Theorem \ref{thm:sob-pull-back}.
\end{proof}

The following theorem states that the multi-valued inverse of a QR map is a quasiminimizer of the $n$-energy and the parametrized volume, respectively, which are defined for $u\in W^{1,n}_{loc}(\Omega,\mathcal A_d(\R^n))$ and $V\subset \Omega$ as follows.

\begin{align*}
E_n(u;V):=\int_V \|Du\|^n\ud x,\quad
A(u;V):=\int_V\J u\ \ud x.
\end{align*}

\begin{theorem}\label{thm:quasimin}
Let $\Omega\subset\R^n$ be a domain and $f:\Omega\to \R^n$ a proper QR map. If $u\in W_{loc}^{1,n}(f(\Omega),\mathcal{A}_d(\R^n))$ is such that $u=\minv f$ a.e. outside a compact set $K\subset \Omega$, then
\begin{align*}
E_n(\minv f;K)&\le d^{n/2-1}K_IE_n(u;K)\\
A(\minv f;K) &\le K_IA(u;K)
\end{align*}
\end{theorem}
Recall that $K_I=K_I(f)$ is the inner distortion of $f$.
\begin{proof}
Let $\tau$ be an $(n-1)$-form on $\mathcal A_d(\R^n)$ such that $\omega=\ud \tau$. Since $u=\minv f$ a.e. on $\Omega\setminus K$, we have that $u^*\omega=\minv f^*\omega$ and $u^*\tau=\minv f^*\tau$ a.e. on $\Omega\setminus K$. 

Set $K_\varepsilon:=\{x:\ \dist(K,x)\le \varepsilon\}$, and note that $K_\varepsilon$ is a finite perimeter set for a.e. $\varepsilon$. Moreover $u^*\tau=\minv f^*\tau$ a.e. on $\partial^*K_\varepsilon$ for a.e. $\varepsilon$ by a Fubini-type argument. Thus, by a generalized Stoke's theorem, we have 
\begin{align}\label{eq:gen-stokes}
\int_{K_\varepsilon} u^*\omega=\int_{K_\varepsilon}\ud(u^*\tau)=\int_{\partial^*K_\varepsilon}u^*\tau = \int_{\partial^*K_\varepsilon}\minv f^*\tau=\int_{K_\varepsilon} \ud(\minv f^*\tau)=\int_{K_\varepsilon} \minv f^*\omega
\end{align}
for a.e. $\varepsilon>0$. (Note that we use Theorem \ref{thm:sob-pull-back} in the first and last equalities above.) Indeed, $\bb{K_\varepsilon}$ is a normal $n$-current on $\R^n$ and $\partial \bb{K_\varepsilon}=\bb{\partial^* K_\varepsilon}$ whenever $K_\varepsilon$ has finite perimeter, and in this case we have by \cite[Theorem 1.1]{iko23} the following: for each $v\in W^{1,n}(f(\Omega),\mathcal A_d(\R^n))$ the identity $\partial(v_\ast\bb{K_\varepsilon})=v_\ast(\partial\bb{K_\varepsilon})=v_\ast(\bb{\partial^*K_\varepsilon})$ holds a.e. $\varepsilon$. This implies
\begin{align*}
\int_{K_\varepsilon}\ud(v^*\beta)=\int_{\partial^*K_\varepsilon}v^*\beta
\end{align*}
for smooth $(n-1)$-forms $\beta$ on $\mathcal A_d(\R^n)$ (cf. also \cite[Lemma 5.3]{iko23}), which in turn implies \eqref{eq:gen-stokes} for a.e. $\varepsilon$. Using \eqref{eq:gen-stokes} we may estimate

\begin{align*}
E_n(\minv f;K)&\le \int_{K_\varepsilon}\|D\minv f\|^n\ud x\le d^{n/2-1}K_I\int_{K_\varepsilon}\star \minv f^*\omega=d^{n/2-1}K_I\int_{K_\varepsilon}\minv f^*\omega\\
&=d^{n/2-1}K_I\int_{K_\varepsilon}u^*\omega\le d^{n/2-1}K_IE_n(u;K_\varepsilon)
\end{align*}
for a.e. $\varepsilon>0$. The last inequality above follows from the pointwise inequalities $\star u^*\omega\le \J u\le \|Du\|^n$, cf. Lemma \ref{lem:star-vs-jacobian}. Since $\leb^n(K_\varepsilon\setminus K)\to 0$ as $\varepsilon\to 0$, we obtain 
\begin{align*}
E_n(\minv f;K)\le d^{n/2-1}K_IE_n(u;K).
\end{align*}

Similarly, 
\begin{align*}
A(\minv f;K)\le \int_{K_\varepsilon} \J \minv f\ud x\le K_I\int_{K_\varepsilon}\minv f^*\omega=K_I \int_{K_\varepsilon}u^*\omega\le K_IA(u;K_\varepsilon)
\end{align*}
(cf. Remark \ref{rmk:jacobian-vs-star} and Lemma \ref{lem:star-vs-jacobian}) yielding 
\begin{align*}
A(\minv f;K)\le K_IA(u;K)
\end{align*}
at the limit $\varepsilon\to 0$. This completes the proof.
\end{proof}

\begin{remark}
The use of pull-back theory and the generalized Stokes theorem could also be used to give another proof of Theorem \ref{thm:higher-integrability}. Indeed, it yields a reverse Hölder inequality without a Caccioppoli inequality, and the higher integrability is then a standard consequence of a theorem of Gehring. We omit the details.
\end{remark}

We end this section with a proof of Corollary \ref{cor:QR-curve}, which is essentially a consequence of Theorems \ref{thm:higher-integrability} and \ref{thm:quasimin}.

\begin{proof}[Proof of Corollary \ref{cor:QR-curve}]
Claim (i) follows immediately from Theorem \ref{thm:higher-integrability}, Corollary \ref{cor:barycenter}(i) and the observation that $g_f=b\circ \minv f$.

To prove (ii), let $h\in W^{1,n}_{loc}(f(\Omega),\R^n)$ be such that $\{h\ne g_f\}$ is essentially contained in a compact subset $K\subset f(\Omega)$. A straightforward calculation shows that if $R>0$, then the multivalued map $R\minv f$, defined by $R\minv f(x)=\sum_{y\in f\inv(x)}\iota(f,x)\bb{Rx}$, is an energy quasi-minimizer with the same constant as $\minv f$ and $\|D(R\minv f)\|=R\|D\minv f\|$. Now consider the multivalued function 
\begin{align*}
    h_R=\mathcal L\inv\circ (R\cdot h,\pi_\mathcal Z\circ\minv f)
\end{align*}
Then $\{h_R\ne \minv f\}$ is essentially contained in $K$, and by Theorem \ref{thm:quasimin} and the discussion above we have 
\begin{align*}
\int_KR^n\|D\minv f\|^n\ud x\le d^{n/2-1}K_I\int_K\|Dh_R\|^n\ud x.
\end{align*}
By Corollary \ref{cor:barycenter}(ii) this implies 
\begin{align*}
\int_Kd^{n/2}R^n\|Dg_f\|^n\le d^{n/2-1}K_I\int_K(dR^2\|Dh\|^2+\|D(\pi_\mathcal Z\circ\minv f)\|^2)^{n/2}\ud x.
\end{align*}
Dividing by $R^n$ and taking the limit $R\to +\infty$ we obtain
\begin{align*}
\int_K\|Dg_f\|^n\ud x\le d^{n/2-1}K_I\int_K\|Dh\|^n\ud x.
\end{align*}

\end{proof}

\section{Multivalued inverse of QR map as a QC map}\label{sec:QC}
Throughout this section we assume that $f:\Omega\to\R^n$ is a proper quasiregular map of degree $d$. The main result of this section is the geometric quasiconformality of the multi-valued inverse onto its image, Theorem \ref{thm:geom-QC}, whose proof is presented in Section \ref{sec:geom-QC}.

\subsection{Image of multivalued inverse}\label{sec:Omega_f-prop}

Consider the space $\Omega_f:=\minv f(f(\Omega))$ equipped with the metric $d_\mathcal A$ from $\mathcal A_d(\Omega)$. We moreover equip $\Omega_f$ with its Hausdorff $n$-measure $\Ha^n$. Since $\minv f:f(\Omega)\to \Omega_f$ is a homeomorphism, $\Omega_f$ is an open $n$-manifold. Moreover, $\Ha^n$ is locally finite by the area-formula.

\begin{lemma}\label{lem:basic-prop-of-Omega_f}
Let $\Ha^n$ be the Hausdorff $n$-measure on $\Omega_f$. Then
\begin{itemize}
\item[(i)] $\Omega_f$ is $n$-rectifiable;
    \item[(ii)] $\Ha^n$ is locally finite;
    \item[(iii)] $\Ha^n(\minv f(f(B_f)))=0$;
\end{itemize}
\end{lemma}
\begin{proof}
All three claims follow from Proposition \ref{prop:area-formula}.
\end{proof}

Next, we establish the upper Ahlfors $n$-regularity of the Hausdorff measure on $\Omega_f$. 
\begin{proposition}\label{prop:ball-meas-upper-bound}
For any compact set $K\subset \Omega_f$ there exists $r_K>0$ such that
\begin{align*}
    \Ha^n(B_{\Omega_f}(z,r))\le \omega_nd^{n/2}K_IK_O r^n,\quad z\in K,\ r<r_K.
\end{align*}
\end{proposition}
\begin{proof}
For any $z\in K$, let $r_z>0$ be such that
\[
(\minv f\circ f)\inv B(z,r)=\bigcup_{x\in(\minv f\circ f)\inv(z)}U_{\minv f\circ f}(x,r)
\]
and $U_{\minv f\circ f}(x,r)$ is a normal neighbourhood of $x$ for each $x\in(\minv f\circ f)\inv(z)$. Let $r_K$ be the Lebesgue number of the cover $\{B(z,r_z)\}_{z\in K}$. For any $z\in K$ there exists $z_0\in K$ such that 
$B(z,r)\subset B(z_0,r_{z_0})$ for $r<r_K$. Note that
\begin{align*}
U_{\minv f\circ f}(x,r)\subset B(x,r),\quad x\in(\minv f\circ f)\inv(z).
\end{align*}
Indeed, for any $y\in U_{F\circ f}(x,r)$ we have $$|y-x|\le d(\minv f(f(y)),F\minv f(f(x))=d(\minv f\circ f(y),z)<r.$$ Now the area formula \eqref{eq:area-formula} yields
\begin{align*}
\Ha^n(B(z,r))&=\int_{\minv f\inv B(z,r)}\J \minv f\ud y\le K_IK_O\int_{\minv f\inv B(z,r)}(f_\ast(|Df|^{-2}))^{n/2}\ud y\\
&\le d^{n/2-1}K_IK_O\int_{\minv f\inv B(z,r)}f_\ast(|Df|^{-n})\ud y\\
&=d^{n/2-1}K_IK_O\int_{f\inv(\minv f\inv B(z,r))} |Df|^{-n}\J f\ud x\\
&=d^{n/2-1}K_IK_O\sum_{x\in (\minv f\circ f)\inv(z)}\int_{U_{\minv f\circ f}(x,r)} \frac{\J \minv f}{|Df|^n}\ud x\\
&\le d^{n/2-1}K_IK_O\sum_{x\in (\minv f\circ f)\inv(z)}|B(x,r)|\le \omega_nd^{n/2}K_IK_Or^n.
\end{align*}
\end{proof}

\begin{remark}
Using Proposition \ref{prop:hausd-meas-est} (or arguing along the lines in the proof above of Proposition \ref{prop:ball-meas-upper-bound}) we can obtain the lower bound
\begin{align*}
\Ha^n(B(z,r))\ge \frac 1{K_Od}\sum_{x\in (\minv f\circ f)\inv(z)}|U_{\minv f\circ f}(x,r)|,\quad z\in K,\ r<r_K.
\end{align*}
A lower bound for $\frac{|U_{\minv f\circ f}(x,r)|}{r^n}$, together with the LLC-property, would imply a local $(1,1)$-Poincar\'e inequality for  $\Omega_f$ by a deep result of Semmes \cite{sem96}.
\end{remark}
In the absence of the lower bound we instead establish, in the next proposition, an infinitesimal $n$-Poincar\'e inequality for $\Omega_f$. This is a weaker qualitative notion compared to $p$-Poincar\'e inequalities, but it is sufficient in order to obtain the geometric quasiconformality of the multi-valued inverse.

A metric measure space $X=(X,d,\mu)$ is said to satisfy an infinitesimal $p$-Poincar\'e inequality ($p$-PI for short), if 
\begin{align}\label{eq:inf-n-PI}
|Dg|_p=\Lip g\quad \mu\textrm{-a.e. on }X
\end{align}
for every $g\in \LIP(X)$. Here $|Dg|_p$ denotes the minimal $p$-weak upper gradient of $g$. 

\begin{proposition}\label{prop:inf-n-PI}
The space $\Omega_f=(\Omega_f,d_\mathcal A,\Ha^n)$ satisfies \eqref{eq:inf-n-PI} with $p=n$ for every $h\in \LIP(\Omega_f)$.
\end{proposition}
In particular, $\Omega_f$ is an $n$-dimensional differentiability space, cf. \cite[Theorem 1.1]{teri-toni-enrico}. In the proof we use the terminology and techniques from \cite{teri-syl24}. Since these are used only in this proof we do not recall them here but refer the reader instead to \cite{teri-syl24}.
\begin{proof}
The map $F=\minv f\circ f:\Omega\to \Omega_f$ is a proper branched cover of degree $d$. It is almost everywhere metrically differentiable and satisfies Lusin's properies $(N)$ and $(N\inv)$. Moreover $F$ also satisfies the modulus inequality
\begin{align*}
\Mod_n\Gamma\le C\Mod_nF(\Gamma)
\end{align*}
for path families in $\Omega$, cf. Proposition \ref{prop:QC} and \cite[Lemma 3.6]{iko-luc-pas24}.
 
By Lemma \ref{lem:basic-prop-of-Omega_f} we have that $\Ha^n(F(B_f))=0$. For every $z_0\in \Omega_f\setminus F(B_f)$ there is $r>0$ so that $F_{0}=F|_{U_F(x_0,r)}: U_{F}(x_0,r)\to B(z_0,r)=:B_0$ is an analytically quasiconformal homeomorphism for every $x_0\in F\inv(z_0)$. Since $d_\mathcal A(F_x(y),F_x(y'))\ge |y-y'|$ we obtain that $F_{0}\inv:B_0\to U_0:=U_{F}(x_0,r)$ is 1-Lipschitz. 

We claim that $(B_0,F_{0}\inv)$ is both a Cheeger chart and an $n$-weak chart. Since $F_0$ is metrically differentiable everywhere and has $(N)$ and $(N\inv)$, it is not difficult to see that $(B_0,F_0)$ is a Cheeger chart.\footnote{Indeed, a direct calculation yields that $(B_0,F_0\inv)$ is a weak Cheeger chart. Since $B_0$ is open, it follows that $(B_0,F_0\inv)$ is a Cheeger chart. See \cite{iko-luc-pas24} for the terminology.} It also follows that, if $(B_0,F_0)$ is $n$-independent (in the sense of \cite{teri-syl24}) then it is maximal, and thus an $n$-weak chart. 

Observe that, for $\Mod_n$-a.e. $\gamma$ in $\Omega$ we have
\begin{align}\label{eq:metr-sp}
|(F_0\circ\gamma)_t'|=\md_{\gamma_t}F_0(\gamma_t')\quad\textrm{a.e. }t.
\end{align}
Let $\Gamma_0$ be the $\Mod_n$-null path family where \eqref{eq:metr-sp} fails. Given $\xi\in (\R^n)^*$ with norm 1 and $v\in S^{n-1}$ such $\xi(v)=1$, let $\Gamma_v$ be the path family of curves in $U_0$ with $\gamma_t'=v$. Then $0<\Mod_n(\Gamma_v\setminus \Gamma_0)\le C\Mod_nF(\Gamma_v\setminus \Gamma_0)$ and, for $\Mod_n$-a.e. curve $\gamma$ in $F(\Gamma_v\setminus\Gamma_0)$, we have 
\begin{align*}
|\gamma_t'|=\md_{F_0\inv(\gamma_t)}F_0(v),\quad (F_0\inv\circ\gamma)_t'=v\quad\textrm{a.e. }t.
\end{align*}
Thus
\begin{align*}
\frac{\xi((F_0\inv\circ\gamma)_t')}{|\gamma_t'|}=\frac{1}{\md_{F_0\inv(\gamma_t)}F_0(v)}\ge \frac{1}{|DF_0|(F_0\inv(\gamma_t))}.
\end{align*}
This implies
\begin{align*}
|D(\xi\circ F_0\inv)|_n(x)\ge \frac{1}{|DF_0|(F_0\inv(x))}\quad\textrm{a.e. }x\in B_0.
\end{align*}
Consequently $F_0\inv$ is $n$-independent in $B_0$, completing the proof that $(B_0,F_0\inv)$ is a $n$-weak chart. 

Now since $(B_0,F_0\inv)$ is an $n$-dimensional Cheeger chart and $n$-weak chart, it follows from \cite[Theorem 1.11]{teri-syl24} and \cite[Theorem 1.3]{teri-toni-enrico} that $\Lip g=|Dg|_n$ $\Ha^n$-a.e. on $\Omega_f$ for any $g\in \LIP(\Omega_f)$. This completes the proof.
\end{proof}

\subsection{Quasiconformality of multi-valued inverse}\label{sec:geom-QC}
It is not difficult to show that $\minv f$ is metrically quasiconformal, i.e. $H(\minv f,x)\le H$ for all $x\in f(\Omega)$, where $\displaystyle H(\minv f,x)=\limsup_{r\to 0}\frac{L_{\minv f}(x,r) }{l_{\minv f}(x,r)}$ and 
\begin{align*}
L_{\minv f}(x,r)&=\sup\{d_\mathcal A(\minv f(x),\minv f(y)):\ d(x,y)\le r\}\\
l_{\minv f}(x,r)&=\sup\{d_\mathcal A(\minv f(x),\minv f(y)):\ d(x,y)\ge r\}.
\end{align*}
In fact we have the following.
\begin{proposition}
The multi-valued inverse $\minv f$ of $f$ satisfies, for small $r>0$:
\begin{align*}
    H_{\minv f}(y,r)^2\le \sum_{x\in f\inv(y)}H_f^*(x,r)^2,\quad y\in f(\Omega).
\end{align*}
\end{proposition}
Here $\displaystyle H^*(f,x)=\limsup_{r\to 0}\frac{L_{f}^*(x,r) }{l_{f}^*(x,r)}$ and 
\begin{align*}
L_{f}^*(x,r)&=\sup\{d_\mathcal A(\minv f(x),\minv f(y)):\ y\in U_f(x,r)\}\\
l_{f}^*(x,r)&=\sup\{d_\mathcal A(\minv f(x),\minv f(y)):\ y\in \partial U_f(x,r)\}.
\end{align*}

\begin{proof}
Denote $f\inv(y)=\{x_1,\ldots,x_m\}$ where $\iota(f,x_1)+\cdots+\iota(f,x_m)=d$. Let $r$ be small enough so that $f\inv B(y,r)=\bigcup_j^mU_f(x_j,r)$ and $U_j:=U_f(x_j,r)$ is a normal domain of $x_j$, $j=1,\ldots,m$. For $z\in B(y,r)$ we have $f\inv(z)=\bigcup_j^mf\inv(z)\cap U_j$ and
\begin{align*}
d_\mathcal A^2(\minv f(y),\minv f(z))= \sum_j^m\sum_{x'\in f\inv(z)\cap U_j}|x_j-x'|^2\le \sum_j^m\iota(f,x_j)L^*_f(x_j,r)^2.
\end{align*}
This implies $L_{\minv f}(y,r)^2\le f_\ast L^*_f(\cdot,r)(y)$. Similarly, for $z\in \partial B(y,r)$,
\begin{align*}
d_\mathcal A^2(\minv f(y),\minv f(z))= \sum_j^m\sum_{x'\in f\inv(z)\cap U_j}|x_j-x'|^2\ge \sum_j^m\iota(f,x_j)l^*_f(x_j,r)^2,
\end{align*}
implying $l_{\minv f}(y,r)^2\ge f_\ast(l^*_f(\cdot,r)^2)$. Together these estimates yield the claim. 
\end{proof}
The next proposition states that $\minv f$ is analytically quasiconformal.
\begin{proposition}\label{prop:QC}
The multi-valued inverse $\minv f:f(\Omega)\to \Omega_f$ of $f$ is an analytically quasiconformal homeomorphism, i.e.
\begin{align*}
\Mod_n\Gamma\le K_IK_O\Mod_n\minv f(\Gamma)
\end{align*}
for every path family $\Gamma$ in $f(\Omega)$. Moreover
\begin{align}\label{eq:jacob-vs-H}
(f_\ast(|Df|^{-2}))^{n/2}\le \J \minv f\le {K_IK_O}(f_\ast(|Df|^{-2}))^{n/2}\quad \textrm{almost everywhere.}    
\end{align}
\end{proposition}

\begin{proof}
Denote $H(y):=(f_\ast(|Df|^{-2}))^{1/2}$ and observe that 
\[
\lim_{h\to 0}\frac{d_\mathcal A(\minv f(y+hv),\minv f(y))}{|h|}=\md_y\minv f(v) \quad\textrm{ a.e. }y\in f(\Omega)
\]
for each $v\in \mathbb S^{n-1}$. By considering the family $\Gamma_v=\{\gamma\subset f(\Omega): \gamma'=v\}$, it follows from a Fubini type argument and Theorem \ref{thm:wug} that 
\begin{align}\label{eq:metric-diff-est}
H(y)\le \md_y\minv f(v)\le (K_IK_O)^{1/n}H(y)\quad \textrm{ for all } v\in \mathbb S^{n-1}.
\end{align}
This implies \eqref{eq:jacob-vs-H}. The remaining claim follows from \eqref{eq:metric-diff-est}. Let $\Gamma$ be a path family in $f(\Omega)$ and let $\rho$ be admissible for $\minv f(\Gamma)$. If $\Gamma_0$ is the exceptional family in Theorem \ref{thm:wug} then $\Mod_v\Gamma_0=0$ and, for any $\gamma\in \Gamma\setminus \Gamma_1$ we have 
\begin{align*}
1\le \int_{\minv f\circ\gamma}\rho\ud s&=\int_0^1\rho(\minv f(\gamma_t))|(\minv f\circ\gamma)_t'|\ud t\le (K_IK_O)^{1/n}\int_0^1\rho\circ \minv f(\gamma_t)H(\gamma_t)|\gamma_t'|\ud t\\
&=(K_IK_O)^{1/n}\int_\gamma(\rho\circ \minv f)\cdot H\ud s.
\end{align*}
Thus $(K_IK_O)^{1/n}(\rho\circ \minv f)H$ is admissible for $\Gamma\setminus \Gamma_0$ implying
\begin{align*}
\Mod_n\Gamma&=\Mod_n\Gamma\setminus\Gamma_0\le K_IK_O\int_{f(\Omega)}\rho^n\circ \minv f H^n\ud y\\
&\le K_IK_O\int_{f(\Omega)}\rho^n\circ \minv f\J \minv f\ud y=K_IK_O\int_{\Omega_f}\rho^n\ud\Ha^n.
\end{align*}
Taking infimum over $\rho$ yields $\Mod_n\Gamma\le K_IK_O\Mod_n\minv f(\Gamma)$ and completes the proof.
\end{proof}

The other inequality requires more work. It is based on the following proposition and an induction argument involving a similar approximation as in the proof of Theorem \ref{thm:lip-pullback-flat-form}. 

\begin{proposition}\label{prop:inv QC}
Let $h:X\to Y$ be a homeomorphism in $N^{1,n}_{loc}(X,Y)$ with $|Dh|^n\le K\ \J h$ and Lusin's property $(N)$ between $n$-rectifiable spaces satisfying the infinitesimal $n$-PI. If $h\inv \in N^{1,n}_{loc}(Y,X)$ and $h\inv $ has Lusin's property $(N)$, then $|D(h\inv)|^n\le K\J (h\inv)$.
\end{proposition}

\begin{proof}
For a Sobolev map $f\in N^{1,n}(Z,W)$ with Lusin's properties $(N)$ and $(N\inv)$ between $n$-rectifiable spaces satisfying the infinitesimal $n$-PI, the minimal upper gradient $|Df|$ of $f$ coincides $\Ha^n$-a.e. $x\in Z$ with $\|D_xf\|_{op}$, where $D_xf$ is the approximate differential of $f$ at $x$, and the volume Jacobian of \cite{wil12b}  coincides with $\J f=\J(D_xf)$. Thus the inequality $|Df|^n\le K\ \J f$ is equivalent with the modulus inequality 
\[
\Mod_n\Gamma\le K\Mod_nf(\Gamma)
\]
for path families $\Gamma$ in $X$.

Now, by rectifiability and the infinitesimal Poincar\'e inequality we have that Gigli's co-tangent module $L^n(T^*Z)$ coincides with the $n$-integrable sections $\Gamma_n(T^*Z)$ of the co-tangent bundle given by the rectifiable atlas, cf. \cite[Theorems 1.2 and 1.3]{teri-toni-enrico}, $Z=X,Y$. The claim now follows from \cite[Theorem 3.13]{iko-luc-pas24}. Indeed, Using the above identification (and the fact that $h$ has Lusin's properties $(N)$ and $(N\inv)$ we have that the pull-back operator $h^*$ constructed in \cite[Theorem 3.4]{iko-luc-pas24} coincides with $(Dh)^*:T^*Y\to T^*X$, where $Dh:TX\to TY$ is the bundle map given by the a.e. defined approximate differential of $h$. Since
\[
\|Dh\|_{op}\le K\J(Dh),
\]
$(Dh)^*$ is an isomorphism, proving condition (i) in \cite[Theorem 3.13]{iko-luc-pas24}, while condition (ii) in \cite[Theorem 3.13]{iko-luc-pas24} is satisfied since $h\inv \in N_{loc}^{1,p}(Y,X)$.
\end{proof}

\begin{theorem}\label{thm:inv QC}
If $f:\Omega\to \R^n$ is a quasiregular map of finite degree $d>0$, then the multi-valued inverse  satisfies $\minv f\inv \in N_{loc}^{1,n}(\Omega_f,\R^n)$ and 
\begin{align*}
    |D\minv f\inv|\le \frac{(K_IK_O)^{1/n}}{H\circ \minv f\inv},\quad |D\minv f\inv|^n\le K_IK_O\J(\minv f\inv).
\end{align*}
Here
\[
H(y):=(f_\ast(|Df|^{-2})(y))^{1/2},\quad y\in f(\Omega).
\]
\end{theorem}

\begin{proof}
We will prove the claim by induction on the degree of $f$. Denote by $I(k)$ the induction hypothesis: $\minv f\inv\in N^{1,n}_{loc}(U_f,\R^n)$ and $|D\minv f\inv|\le\frac{C(f)}{H\circ \minv f\inv}$ for every quasiregular map $f:U\to \R^n$ of degree $\le k$ from an open set $U\subset \R^n$, where $C(f)=(K_I(f)K_O(f))^{1/n}$. When $k=1$, the claim is classical. Indeed, a locally homeomorphic quasiregular map is quasiconformal \cite[Chapter III, Corollary 3.8]{ric93}, and its inverse is quasiconformal with the correct bound \cite[Chapter II, Corollary 6.5]{ric93}. Thus $I(1)$ holds.

Suppose now that $I(k)$ holds for some $k\in \N$. We will prove that $I(k+1)$ holds. Let $d=k+1$ and let $f:\Omega\to \R^n$ be a quasiregular map of degree $d$. Set $A=\Omega_f\cap \diag\mathcal A_d(\Omega)$ and note that, for all $z_0\in \Omega_f\setminus A$ there exists $x_0\in (\minv f\circ f)\inv(z_0)$ and $r_0>$ such that $\iota(f,x_0)<d$ and the map $$f_{0}:=f|_{U_{\minv f\circ f}(x_0,r_0)}:U_{\minv f\circ f}(x_0,r_0)\to F\inv B(z_0,r_0)$$ is a quasiregular map of degree $\le k$. Denote $U_{0}=U_{\minv f\circ f}(x_0,r_0)$, $B_0=B(z_0,r_0)$, $\minv f_0:F\inv B_0\to \Omega_{f_{0}}$ (the multivalued inverse of $f_0$), and $H_0=\big(f_{0\ast}(|Df_0|^{-2}))^{1/2}$. By induction hypothesis $\minv f_0\inv\in N^{1,n}_{loc}((U_0)_{f_0},\R^n)$ and 
\begin{align*}
    |D\minv f_0\inv|\le \frac{C(f_0)}{H_0\circ \minv f_0\inv}.
\end{align*}

By Proposition \ref{prop:inv QC} we have that $\minv f_0\inv:(U_0)_{f_0}\to F\inv B_0$ is analytically quasiconformal. Thus $\minv f\circ \minv f_0\inv:(U_0)_{f_0}\to B_0$ and $P_0:=(\minv f\circ \minv f_0\inv)\inv:B_0\to (U_0)_{f_0}$ are analytically quasiconformal (again by Proposition \ref{prop:inv QC}, since $P_0$ is 1-Lipschitz). Thus $\minv f\inv|_{B_0}=\minv f_0\inv\circ P_0$ is analytically quasiconformal and moreover
\begin{align*}
|D\minv f\inv|&=|D(\minv f_0\inv\circ P_0)|\le |D\minv f_0\inv|\circ P_0|DP_0|\\
&\le \frac {C(f_0)}{H_0\circ \minv f\inv }\cdot C(f)\frac{H_0\circ \minv f\inv}{C(f_0)H\circ \minv f\inv}=\frac {C(f)}{H\circ \minv f\inv}\quad\textrm{on }B_0.
\end{align*}
By repeating this argument for all points $z_0\in \Omega_f\setminus A$ it follows that $\minv f\inv|_{\Omega_f\setminus A}\in N_{loc}^{1,n}(\Omega_f\setminus A,\R^n)$ with
\begin{align}\label{eq:inv QC}
    |D\minv f\inv|\le \frac{C(f)}{H\circ \minv f\inv} \quad\textrm{a.e. on }\Omega_f\setminus A.
\end{align}

Fix a smooth function $\varphi:\R\to [0,1]$ such that $\varphi=1$ in a neighbourhood of $0$, $\spt\varphi\subset (-1,1)$ and $\LIP(\varphi)\le 2$. For each $\varepsilon>0$ set
\begin{align*}
    G_\varepsilon:=\varphi\Big(\frac{f\circ b-\minv f\inv}{\varepsilon}\Big)f\circ b+ \Big(1-\varphi\Big(\frac{f\circ b-\minv f\inv}{\varepsilon}\Big)\Big)\minv f\inv.
\end{align*}
Note that $G_\varepsilon=f\circ b$ in a neighbourhood of $A$, $G_\varepsilon=\minv f\inv$ in $\Omega_f\setminus B(A,\varepsilon)$, and $G_\varepsilon\in N_{loc}^{1,n}(\Omega_f,\R^n)$.\footnote{Indeed, for all points in $B(A,\varepsilon)\setminus A$, $G_\varepsilon$ is a composition of maps in $W^{1,n}$.}

Writing $G_\varepsilon=\minv f\inv +\varphi\Big(\frac{f\circ b-\minv f\inv}{\varepsilon}\Big)(f\circ b-\minv f\inv)$ we have 
\begin{align*}
|DG_\varepsilon|\le &|D\minv f\inv|+|f\circ b-\minv f\inv|\Big|D\varphi\circ\Big(\frac{f\circ b-\minv f\inv}{\varepsilon}\Big)\Big|\\
&+\varphi\Big(\frac{f\circ b-\minv f\inv}{\varepsilon}\Big)|D(f\circ b-\minv f\inv)|\\
\le & |D\minv f\inv|+|f\circ b-\minv f\inv|\cdot 2\chi_{(-1,1)}\Big(\frac{f\circ b-\minv f\inv}{\varepsilon}\Big) \frac{|D(f\circ b-\minv f\inv)|}{\varepsilon}\\
&+\varphi\Big(\frac{f\circ b-\minv f\inv}{\varepsilon}\Big)|D(f\circ b-\minv f\inv)|\\
\le &|D\minv f\inv|+3|D(f\circ b-\minv f\inv)|.
\end{align*}
Moreover
\begin{align*}
G_\varepsilon-\minv f\inv=\varphi\Big(\frac{f\circ b-\minv f\inv}{\varepsilon}\Big)(f\circ b-\minv f\inv)
\end{align*}
which converges to 0 locally uniformly. Together with the uniform bound $|DG_\varepsilon|\le |D\minv f\inv |+3|D(f\circ b-\minv f\inv)|$ (independent of $\varepsilon$), this implies that $\minv f\inv\in N_{loc}^{1,n}(\Omega_f,\R^n)$. Since $\Ha^n(A)=0$, \eqref{eq:inv QC} implies the bound on of the minimal upper gradient in the claim. The inequality $|D\minv f\inv|^n\le K_IK_O\J (\minv f\inv)$ follows from Proposition \ref{prop:inv QC}. This completes the induction step and the proof.
\end{proof}

We close the section with the proof of Theorem \ref{thm:geom-QC}.
\begin{proof}[Proof of Theorem \ref{thm:geom-QC}]
The first claim (1) is contained in Lemma \ref{lem:basic-prop-of-Omega_f}, Proposition \ref{prop:ball-meas-upper-bound} and Proposition \ref{prop:inf-n-PI}. Indeed, $n$-rectifiability follows from the fact that $\Omega_f$ is the image of a Sobolev map from an open set in $\R^n$ with property $(N)$, upper Ahlfors $n$-regularity proved in Proposition \ref{prop:ball-meas-upper-bound}, and the infinitesimal Poincar\'e inequality in Proposition \ref{prop:inf-n-PI}.

The second claim (2) follows by combining Proposition \ref{prop:QC} and Theorem \ref{thm:inv QC}. Indeed, the analytic quasiconformality of the inverse $\minv f\inv$, given by Theorem \ref{thm:inv QC}, is equivalent with the modulus inequality 
\[
\Mod_n\minv f(\Gamma)\le K_IK_O\Mod_n \Gamma
\]
for path families in $f(\Omega)$.
\end{proof}

\appendix
\section{Maximal Hilbertian factor of Almgren space}\label{sec:appendix} Throughout this section we fix integers $d,n\ge 1$. By convention $\mathcal A_0(\R^n)=\{0\}$. Recall the barycenter map $b:\mathcal A_d(\R^n)\to \R^n$ \eqref{eq:barycenter}, and define
\begin{align*}
\mathcal Z_d:=\{\bb{\bar x}\in \mathcal A_d(\R^n):\ b(\bb{\bar x})=0\}.   
\end{align*}
We let $b_d:(\R^n)^d\to\R^n$ be the linear map $\bar x\mapsto \frac{x_1+\cdots+x_d}{d}$ and note that $b(\bb{\bar x})=b_d(\bar x)$. Given $\bb{\bar x}=\bb{x_1,\ldots,x_d}\in\mathcal A_d(\R^n)$, set $\bb{\bar x-b(\bar x)}:=\bb{x_1-b(\bar x),\ldots,x_d-b(\bar x)}$ and note that 
\[
\bb{\bar x-b(\bar x)}\in \mathcal Z_d\quad\textrm{for all }\bb{\bar x}\in\mathcal A_d(\R^n).
\]
Denote $\pi_{\mathcal Z}:\mathcal A_d(\R^n)\to \mathcal Z_d$ the map $\bb{\bar x}\mapsto \bb{\bar x-b(\bar x)}$.

The factorization in Proposition \ref{prop:splitting} below reflects the fact that the diagonal of $(\R^n)^d$ is the invariant set of the action of $S_d$. Indeed, it is easy to see that $\mathcal L$ maps $\diag\mathcal A_d(\R^n)$ to the first factor. 
\begin{proposition}\label{prop:splitting}
    Equip $\R^n\times \mathcal{Z}_d$ with the Euclidean product metric. Then the map $\mathcal L:\mathcal A_d(\R^n)\to \R^n\times \mathcal Z_d$ given by
\begin{align*}
\mathcal L(\bb{\bar x})=(\sqrt db(\bb{\bar x}),\bb{\bar x-b(\bb{\bar x})})
\end{align*}
is an isometric homeomorphism. Moreover
\begin{itemize}
    \item[(i)] $\mathcal Z_d$ does not contain any geodesic line (i.e. isometry $\R\to \mathcal Z_d$).
    \item[(ii)] there exists an open map $F:\mathcal A_{d-1}(\R^n)\to \mathcal Z_d$ such that $\#F\inv(p)\le d$ for all $p\in \mathcal Z_d$, and $\ell(F\circ\gamma)=\ell(\gamma)$ for all curves $\gamma$ in $\mathcal A_{d-1}(\R^n)$.
\end{itemize}
\end{proposition}

We remark that (i) and (ii) are not needed anywhere in this paper, but are included out of independent interest. Property (i) states that $\R^n$ is the maximal Hilbert space factor of $\mathcal A_d(\R^n)$, while property (ii) says that $\mathcal Z_d$ admits a branched covering from $\mathcal A_{d-1}(\R^n)$.

Before providing the proof we record the following immediate corollary, which is used in Section \ref{sec:QR-curve}.
\begin{corollary}\label{cor:barycenter}
\begin{itemize}
    \item[(i)] The barycenter map \eqref{eq:barycenter} is $1/\sqrt d$-Lipschitz;
    \item[(ii)] If $f:U\to \mathcal A_d(\R^n)$ is approximately differentiable at $x$, then we have the inequalities
\begin{align*}
\|D_xf\|\le \sqrt{d \|D_x(b\circ f)\|^2+\|D_x(\pi_\mathcal Z\circ f)\|^2},\\
\max\{\sqrt d \|D_x(b\circ f)\|,\|D_x(\pi_\mathcal Z\circ f)\|\}\le \|Df\|.
\end{align*}
\end{itemize} 
\end{corollary}
\begin{proof}
For each $\bb{\bar x},\ \bb{\bar y}\in \mathcal A_d(\R^n)$ we have
\begin{align*}
d|b(\bb{\bar y})-b(\bb{\bar x})|^2\le d|b(\bb{\bar y})-b(\bb{\bar x})|^2+d^2_{\mathcal A}(\bb{\bar x-b(\bar x)},\bb{\bar yb(\bar y)})=d_{\mathcal A}^2(\bb{\bar y},\bb{\bar x}).
\end{align*}
This proves (i). To see (ii), simply apply the isometric splitting of Proposition \ref{prop:splitting} to the observation that the operator norm of a linear map $L:\R^m\to \R^a\times \R^b$ satisfies 
\[
\max\{\|P_a\circ L\|,\|P_b\circ L\|\}\le \|L\|\quad\textrm{and}\quad \|L\|\le  \sqrt{\|P_a\circ L\|^2+\|P_b\circ L\|^2}
\]
\end{proof}

\begin{proof}[Proof of Proposition \ref{prop:splitting}]
Let $\bb{\bar x}=\bb{x_1,\ldots,x_d},\bb{\bar y}=\bb{y_1,\ldots,y_d}\in \mathcal A_d(\R^n)$. By relabeling we may assume that 
\[
d_\mathcal A^2(\bb{\bar x},\bb{\bar y})^2=\sum_j^d|x_j-y_j|^2.
\]
Then also 
\begin{align}\label{eq:isom}
d_\mathcal A(\bb{x-b(\bar x)},\bb{\bar y-b(\bar y)})^2=\sum_j^d|x_j-y_j-b(\bar{x}-\bar{y})|^2.
\end{align}
Indeed, if $\displaystyle \sum_j^d|x_j-y_{\sigma(j)}-b(\bar{x}-\bar{y})|^2<\sum_j^d|x_j-y_j-b(\bar{x}-\bar{y})|^2$ for some $\sigma\in S_d\setminus\{id\}$, expanding the squares and eliminating the terms $|b(\bar x-\bar y)|^2$ from both sides we obtain:
\begin{align*}
\sum_j^d[|x_j-y_{\sigma(j)}|^2-2\langle x_j-y_{\sigma(j)},b(\bar x-\bar y)\rangle]&< \sum_j^d[|x_j-y_{j}|^2-2\langle x_j-y_{j},b(\bar x-\bar y)\rangle]\\
\Longleftrightarrow\\
\sum_j^d|x_j-y_{\sigma(j)}|^2&<\sum_j^d[|x_j-y_{j}|^2,
\end{align*}
since 
\[
\sum_j^d\langle x_j-y_{\sigma(j)},b(\bar x-\bar y)\rangle=\sum_j^d\langle x_j-y_{j},b(\bar x-\bar y)\rangle.
\]
This would be a contradiction, and thus \eqref{eq:isom} is established.

Denoting  $z=(x_1-y_1,\ldots,x_d-y_d)\in (\R^n)^d$, we recall that the average satisfies 
\[
\sum_j^d\|z_j-b(z)\|^2=\sum_j^d\|z_j\|^2-d\|b(z)\|^2,
\]
which yields 
\begin{align*}
d_\mathcal A(\bb{x-b(\bar x)},\bb{\bar y-b(\bar y)})^2=d_\mathcal A^2(\bb{\bar x},\bb{\bar y})^2-d\|b(\bar x)-b(\bar y)\|^2.
\end{align*}
This proves that $\mathcal L$ is an isometry. Since it is clearly a surjection, it is an isometric homeomorphism. Properties (i) and (ii) in the claim are proved in Lemmas \ref{lem:(i)} and \ref{lem:(ii)}, respectively.
\end{proof}

\begin{lemma}\label{lem:(i)}
There does not exist an isometric embedding $\R\to \mathcal Z_d$.
\end{lemma}
In the proof we use the projection map $\pi:(\R^n)^d\to \mathcal A_d(\R^n)$ which is 1-BLD, i.e. length preserving, discrete and open (cf. \cite{lui17} for the terminology). Note in particular that $\pi|_{\ker(b_d)}:\ker(b_d)\to \mathcal Z_d$ is also 1-BLD. 
\begin{proof}
Suppose $\gamma: \R\to \mathcal Z_d$ is an isometry. Since $\pi|_{\mathcal Z_d}$ is 1-BLD, by \cite[Lemma 2.8]{lui17} there exist an affine geodesic line $\tilde \gamma(t)=a+tv$ in $\ker(b_d)\subset (\R^n)^d$ so that $\pi \circ\tilde\gamma=\gamma$. It follows that $a,v\in \ker(b_d)$, and $|v|=1$. Given $t\in \R$ and $\sigma\in S_d$ we have
\begin{align*}
    (2t)^2=d_\mathcal A(\gamma_{-t},\gamma_{t})^2\le \sum_j^d|\tilde\gamma_j(-t)-\tilde\gamma_{\sigma(j)}(t)|^2=\sum_j^d|t(v_j+v_{\sigma(j)})-(a_j-a_{\sigma(j)})|^2.
\end{align*}
Diving by $t^2$ and taking the limit $t\to +\infty$ we obtain
\begin{align*}
4\le \sum_j^d|v_j+v_{\sigma(j)}|^2=\sum_j^d[|v_j|^2+2\langle v_j,v_{\sigma(j)}\rangle+|v_{\sigma(j)}|^2]=2+2\sum_j^d\langle v_j,v_{\sigma(j)}\rangle,
\end{align*}
or
\begin{align*}
1\le \sum_j^d\langle v_j,v_{\sigma(j)}\rangle,
\end{align*}
which can only hold if $v_j=v_{\sigma(j)}$ for $j=1,\ldots,d$. Since $\sigma\in S_d$ is arbitrary this implies that $v\in \diag(\R^n)^d$, but since $v\in \ker(b_d)$ we obtain $v=0$, a contradiction. This completes the proof.
\end{proof}

\begin{lemma}\label{lem:(ii)}
There exists $\lambda\in (0,1)$ such that the map $F:\mathcal A_d(\R^n)\to \mathcal Z_{d+1}$ given by
\begin{align*}
F(\bb{\bar x})=\bb{x_1-\lambda b(\bar x),\ldots,x_d-\lambda b(\bar x),d(\lambda-1)b(\bar x)},\quad \bb{\bar x}=\bb{x_1,\ldots,x_d},
\end{align*}
is open, $\#F\inv(p)\le d+1$ for all $p\in \mathcal Z_{d+1}$, and $\ell(F\circ\gamma)=\ell(\gamma)$ for all curves $\gamma$ in $\mathcal A_d(\R^n)$. Moreover, $F|_{\mathcal Z_d}$ is an isometric embedding.
\end{lemma}
\begin{proof}
Given $\lambda\in [0,1]$, consider the linear map $A_\lambda:(\R^n)^d\to (\R^n)^{d+1}$ given by
\begin{align*}
    A_\lambda(\bar x)=(x_1-\lambda b(\bar x),\ldots,x_d-\lambda b(\bar x),d(\lambda-1)b(\bar x)),\quad\bar x=(x_1,\ldots, x_d).
\end{align*}
Note that $b_{d+1}\circ A_\lambda =0$, which implies that $\operatorname{Im}(A_\lambda)\subset \ker(b_{d+1})$. Indeed, 
\begin{align*}
(d+1)b_{d+1}(A_\lambda (\bar x))&=d(\lambda-1)b_d(\bar x)+\sum_j^d(x_j-\lambda b_d(\bar x))\\
&=d(\lambda-1)b_d(\bar x)+db_{d}(\bar x)-\lambda d b_d(\bar x)=0.
\end{align*}

Moreover we may calculate
\begin{align*}
|A_\lambda(\bar x)|^2&=d^2(\lambda-1)^2|b_d(\bar x)|^2+\sum_j^d|x_j-\lambda b_d(\bar x)|^2\\
&=d^2(\lambda-1)^2|b_d(\bar x)|^2+\sum_j^d[|x_j|^2-2\lambda \langle x_j,b_d(\bar x)\rangle+\lambda^2|b_d(\bar x)|^2]\\
&=|\bar x|^2+|b_d(\bar x)|^2(d^2(\lambda -1)^2-2\lambda d+\lambda^2d).
\end{align*}
Denoting $P(\lambda)=d^2(\lambda -1)^2-2\lambda d+\lambda^2d$ we have $P(0)=d^2$ and $P(1)=-d$. Since $P$ is continuous, it takes the value 0 for some $\lambda \in (0,1)$. We fix this $\lambda$ for the rest of the proof and drop it from the subscript. Thus $A:(\R^n)^d\to (\R^n)^{d+1}$ is an isometric linear map with image $\ker(b_{d+1})$. 

Consequently the map $F:\mathcal A_d(\R^n)\to \mathcal Z_{d+1}$ given by 
\begin{align*}
F(\bb{\bar x})=\bb{A(\bar x)}
\end{align*}
is well-defined and satisfies $\ell(F\circ\gamma)=\ell(\gamma)$ for all curves $\gamma$ in $\mathcal A_d(\R^n)$. Openness follows from the definition of quotient topologies, while for $\bb{\bar x}\in \mathcal Z_d$ we have $F(\bb{\bar x})=\bb{x_1,\ldots,x_d,0}$ which is easily seen to be an isometry. 

It remains to prove that $\#F\inv(p)\le d+1$ for all $p$. If $p=F(\bb{\bar x})$, then $b_d(\bar x)$ can take at most $d+1$ distinct values (corresponding to the unordered components of $p$). But if $F(\bb{\bar x})=F(\bb{\bar y})$ and $b_d(\bar x)=b_d(\bar y)$, then $\bar x=\sigma \bar y$ for some $\sigma\in S_d$. This means that there can be at most $d+1$ distinct points $\bb{\bar x}\in\mathcal A_d(\R^n)$ so that $F(\bb{\bar x})=p$. 
\end{proof}

\begin{remark}
The proof shows that in fact $\#F\inv(p)$ is the number of distinct points in $p=\bb{p_1,\ldots,p_{d+1}}$.
\end{remark}

\bibliographystyle{plain}
\bibliography{abib}

@preamble{"\def\cprime{$'$}"}

@unpublished{iko23,
    author = {Ikonen, Toni},
    title = {Pushforward of currents under Sobolev maps},
    note = {arXiv:2303.15003},
    year = {2023}
}

@article {men10,
    AUTHOR = {Menne, Ulrich},
     TITLE = {A {S}obolev {P}oincar\'e{} type inequality for integral
              varifolds},
   JOURNAL = {Calc. Var. Partial Differential Equations},
  FJOURNAL = {Calculus of Variations and Partial Differential Equations},
    VOLUME = {38},
      YEAR = {2010},
    NUMBER = {3-4},
     PAGES = {369--408},
      ISSN = {0944-2669,1432-0835},
   MRCLASS = {49Q20},
  MRNUMBER = {2647125},
MRREVIEWER = {Cristina\ Tarsi},
       DOI = {10.1007/s00526-009-0291-9},
       URL = {https://doi.org/10.1007/s00526-009-0291-9},
}

@book {HKM06,
    AUTHOR = {Heinonen, Juha and Kilpel\"ainen, Tero and Martio, Olli},
     TITLE = {Nonlinear potential theory of degenerate elliptic equations},
      NOTE = {Unabridged republication of the 1993 original},
 PUBLISHER = {Dover Publications, Inc., Mineola, NY},
      YEAR = {2006},
     PAGES = {xii+404},
      ISBN = {0-486-45050-3},
   MRCLASS = {31C45 (35J60 35J70)},
  MRNUMBER = {2305115},
MRREVIEWER = {Piotr\ Haj\l asz},
}

@article {onn-raj09,
    AUTHOR = {Onninen, Jani and Rajala, Kai},
     TITLE = {Quasiregular mappings to generalized manifolds},
   JOURNAL = {J. Anal. Math.},
  FJOURNAL = {Journal d'Analyse Math\'ematique},
    VOLUME = {109},
      YEAR = {2009},
     PAGES = {33--79},
      ISSN = {0021-7670,1565-8538},
   MRCLASS = {30C65 (30L10)},
  MRNUMBER = {2585391},
MRREVIEWER = {Anatoly\ Golberg},
       DOI = {10.1007/s11854-009-0028-x},
       URL = {https://doi.org/10.1007/s11854-009-0028-x},
}

@unpublished{iko-pan24,
    author = {Ikonen, Toni and Pankka, Pekka},
    title = {Liouville's theorem in calibrated geometries},
    year = {2024},
    note = {arxiv: 2410.02722},
}

@article {hei-pan-pry23,
    AUTHOR = {Heikkil\"a, Susanna and Pankka, Pekka and Prywes, Eden},
     TITLE = {Quasiregular curves of small distortion in product manifolds},
   JOURNAL = {J. Geom. Anal.},
  FJOURNAL = {Journal of Geometric Analysis},
    VOLUME = {33},
      YEAR = {2023},
    NUMBER = {1},
     PAGES = {Paper No. 1, 44},
      ISSN = {1050-6926,1559-002X},
   MRCLASS = {53C15 (30C65 58J60)},
  MRNUMBER = {4502762},
MRREVIEWER = {David\ Matthew\ Freeman},
       DOI = {10.1007/s12220-022-01053-4},
       URL = {https://doi.org/10.1007/s12220-022-01053-4},
}

@article {wil12b,
    AUTHOR = {Williams, Marshall},
     TITLE = {Geometric and analytic quasiconformality in metric measure
              spaces},
   JOURNAL = {Proc. Amer. Math. Soc.},
  FJOURNAL = {Proceedings of the American Mathematical Society},
    VOLUME = {140},
      YEAR = {2012},
    NUMBER = {4},
     PAGES = {1251--1266},
      ISSN = {0002-9939,1088-6826},
   MRCLASS = {30L10 (30C65 46E35)},
  MRNUMBER = {2869110},
MRREVIEWER = {Leonid\ V.\ Kovalev},
       DOI = {10.1090/S0002-9939-2011-11035-9},
       URL = {https://doi.org/10.1090/S0002-9939-2011-11035-9},
}

@article {iko-luc-pas24,
    AUTHOR = {Ikonen, Toni and Lu\v{c}i\'c, Danka and Pasqualetto, Enrico},
     TITLE = {Pullback of a quasiconformal map between arbitrary metric
              measure spaces},
   JOURNAL = {Illinois J. Math.},
  FJOURNAL = {Illinois Journal of Mathematics},
    VOLUME = {68},
      YEAR = {2024},
    NUMBER = {1},
     PAGES = {137--165},
      ISSN = {0019-2082,1945-6581},
   MRCLASS = {30L10 (30C65 31E05 46E35 53C23)},
  MRNUMBER = {4720559},
MRREVIEWER = {Matthew\ Romney},
       DOI = {10.1215/00192082-11081290},
       URL = {https://doi.org/10.1215/00192082-11081290},
}

@article {teri-toni-enrico,
    AUTHOR = {Ikonen, Toni and Pasqualetto, Enrico and Soultanis,
              Elefterios},
     TITLE = {Abstract and concrete tangent modules on {L}ipschitz
              differentiability spaces},
   JOURNAL = {Proc. Amer. Math. Soc.},
  FJOURNAL = {Proceedings of the American Mathematical Society},
    VOLUME = {150},
      YEAR = {2022},
    NUMBER = {1},
     PAGES = {327--343},
      ISSN = {0002-9939,1088-6826},
   MRCLASS = {53C23 (46E35 49J52 49Q15)},
  MRNUMBER = {4335880},
MRREVIEWER = {David\ Matthew\ Freeman},
       DOI = {10.1090/proc/15656},
       URL = {https://doi.org/10.1090/proc/15656},
}

@article {kar07,
    AUTHOR = {Karmanova, M. B.},
     TITLE = {Area and co-area formulas for mappings of the {S}obolev
              classes with values in a metric space},
   JOURNAL = {Sibirsk. Mat. Zh.},
  FJOURNAL = {Rossi\u iskaya Akademiya Nauk. Sibirskoe Otdelenie. Institut
              Matematiki im. S. L. Soboleva. Sibirski\u i\ Matematicheski\u
              i\ Zhurnal},
    VOLUME = {48},
      YEAR = {2007},
    NUMBER = {4},
     PAGES = {778--788},
      ISSN = {0037-4474},
   MRCLASS = {46E35 (26B05 26B15 28A75 30C65 46G10 52A38)},
  MRNUMBER = {2355373},
MRREVIEWER = {Jan\ Mal\'y},
       DOI = {10.1007/s11202-007-0064-7},
       URL = {https://doi.org/10.1007/s11202-007-0064-7},
}

@article {teri-syl24,
    AUTHOR = {Eriksson-Bique, Sylvester and Soultanis, Elefterios},
     TITLE = {Curvewise characterizations of minimal upper gradients and the
              construction of a {S}obolev differential},
   JOURNAL = {Anal. PDE},
  FJOURNAL = {Analysis \& PDE},
    VOLUME = {17},
      YEAR = {2024},
    NUMBER = {2},
     PAGES = {455--498},
      ISSN = {2157-5045,1948-206X},
   MRCLASS = {46E36 (26B05 30L99 49J52 53C23)},
  MRNUMBER = {4713106},
       DOI = {10.2140/apde.2024.17.455},
       URL = {https://doi.org/10.2140/apde.2024.17.455},
}

@article {flo50,
    AUTHOR = {Floyd, E. E.},
     TITLE = {Some characterizations of interior maps},
   JOURNAL = {Ann. of Math. (2)},
  FJOURNAL = {Annals of Mathematics. Second Series},
    VOLUME = {51},
      YEAR = {1950},
     PAGES = {571--575},
      ISSN = {0003-486X},
   MRCLASS = {56.0X},
  MRNUMBER = {35007},
MRREVIEWER = {A.\ D.\ Wallace},
       DOI = {10.2307/1969369},
       URL = {https://doi.org/10.2307/1969369},
}

@article {gro85,
    AUTHOR = {Gromov, M.},
     TITLE = {Pseudo holomorphic curves in symplectic manifolds},
   JOURNAL = {Invent. Math.},
  FJOURNAL = {Inventiones Mathematicae},
    VOLUME = {82},
      YEAR = {1985},
    NUMBER = {2},
     PAGES = {307--347},
      ISSN = {0020-9910,1432-1297},
   MRCLASS = {53C15 (32F25 53C57 57R15)},
  MRNUMBER = {809718},
MRREVIEWER = {Yakov\ Eliashberg},
       DOI = {10.1007/BF01388806},
       URL = {https://doi.org/10.1007/BF01388806},
}

@article {pan-onn21,
    AUTHOR = {Onninen, Jani and Pankka, Pekka},
     TITLE = {Quasiregular curves: {H}\"older continuity and higher
              integrability},
   JOURNAL = {Complex Anal. Synerg.},
  FJOURNAL = {Complex Analysis and its Synergies},
    VOLUME = {7},
      YEAR = {2021},
    NUMBER = {3},
     PAGES = {Paper No. 26, 9},
      ISSN = {2524-7581,2197-120X},
   MRCLASS = {30C65 (32A30 53C15)},
  MRNUMBER = {4289854},
MRREVIEWER = {Irina\ G.\ Markina},
       DOI = {10.1007/s40627-021-00086-9},
       URL = {https://doi.org/10.1007/s40627-021-00086-9},
}

@article {IVV02,
    AUTHOR = {Iwaniec, Tadeusz and Verchota, Gregory C. and Vogel, Andrew
              L.},
     TITLE = {The failure of rank-one connections},
   JOURNAL = {Arch. Ration. Mech. Anal.},
  FJOURNAL = {Archive for Rational Mechanics and Analysis},
    VOLUME = {163},
      YEAR = {2002},
    NUMBER = {2},
     PAGES = {125--169},
      ISSN = {0003-9527,1432-0673},
   MRCLASS = {58E15 (30C65 74B20)},
  MRNUMBER = {1911096},
       DOI = {10.1007/s002050200197},
       URL = {https://doi.org/10.1007/s002050200197},
}

@article {pankka20,
    AUTHOR = {Pankka, Pekka},
     TITLE = {Quasiregular curves},
   JOURNAL = {Ann. Acad. Sci. Fenn. Math.},
  FJOURNAL = {Annales Academi\ae\ Scientiarum Fennic\ae. Mathematica},
    VOLUME = {45},
      YEAR = {2020},
    NUMBER = {2},
     PAGES = {975--990},
      ISSN = {1239-629X,1798-2383},
   MRCLASS = {30C65 (30C62 32Q65 53C15 53C56)},
  MRNUMBER = {4112271},
MRREVIEWER = {Riikka\ Kangaslampi},
       DOI = {10.5186/aasfm.2020.4534},
       URL = {https://doi.org/10.5186/aasfm.2020.4534},
}

@article {hei11,
	AUTHOR = {Heinonen, Juha and Keith, Stephen},
	TITLE = {Flat forms, bi-{L}ipschitz parameterizations, and
	smoothability of manifolds},
	JOURNAL = {Publ. Math. Inst. Hautes \'{E}tudes Sci.},
	FJOURNAL = {Publications Math\'{e}matiques. Institut de Hautes \'{E}tudes
	Scientifiques},
	NUMBER = {113},
	YEAR = {2011},
	PAGES = {1--37},
	ISSN = {0073-8301},
	MRCLASS = {30L10 (46E35 57P99 57R10 58A10)},
	MRNUMBER = {2805596},
	MRREVIEWER = {Jeremy T. Tyson},
	DOI = {10.1007/s10240-011-0032-4},
	URL = {https://doi.org/10.1007/s10240-011-0032-4},
}

@article {teripekka,
    AUTHOR = {Pankka, Pekka and Soultanis, Elefterios},
     TITLE = {Pull-back of metric currents and homological boundedness of
              {BLD}-elliptic spaces},
   JOURNAL = {Anal. Geom. Metr. Spaces},
  FJOURNAL = {Analysis and Geometry in Metric Spaces},
    VOLUME = {7},
      YEAR = {2019},
    NUMBER = {1},
     PAGES = {212--249},
      ISSN = {2299-3274},
   MRCLASS = {30L10 (30C65 49Q15 58A25)},
  MRNUMBER = {4036563},
MRREVIEWER = {Matthew\ Romney},
       DOI = {10.1515/agms-2019-0011},
       URL = {https://doi.org/10.1515/agms-2019-0011},
}

@article{kir94,
	author = "Kirchheim, Bernd",
	doi = "10.2307/2160371",
	fjournal = "Proceedings of the American Mathematical Society",
	issn = "0002-9939",
	journal = "Proc. Amer. Math. Soc.",
	mrclass = "28A78",
	mrnumber = "1189747",
	mrreviewer = "G. Freilich",
	number = "1",
	pages = "113--123",
	title = "{Rectifiable metric spaces: local structure and regularity of the {H}ausdorff measure}",
	url = "https://doi.org/10.2307/2160371",
	volume = "121",
	year = "1994"
}

@article{hei00,
	author = "Heinonen, Juha and Koskela, Pekka and Shanmugalingam, Nageswari and Tyson, Jeremy T.",
	coden = "JOAMAV",
	doi = "10.1007/BF02788076",
	fjournal = "Journal d'Analyse Math{\'e}matique",
	issn = "0021-7670",
	journal = "J. Anal. Math.",
	mrclass = "46E40 (30C65 46E35)",
	mrnumber = "1869604 (2002k:46090)",
	mrreviewer = "Shusen Ding",
	pages = "87--139",
	title = "{Sobolev classes of {B}anach space-valued functions and quasiconformal mappings}",
	url = "http://dx.doi.org/10.1007/BF02788076",
	volume = "85",
	year = "2001"
}

@article{hei98,
	author = "Heinonen, Juha and Koskela, Pekka",
	coden = "ACMAA8",
	doi = "10.1007/BF02392747",
	fjournal = "Acta Mathematica",
	issn = "0001-5962",
	journal = "Acta Math.",
	mrclass = "30C65 (46E99)",
	mrnumber = "1654771 (99j:30025)",
	mrreviewer = "M. Yu. Vasil{\cprime}chik",
	number = "1",
	pages = "1--61",
	title = "{Quasiconformal maps in metric spaces with controlled geometry}",
	url = "http://dx.doi.org/10.1007/BF02392747",
	volume = "181",
	year = "1998"
}

@book{hat02,
	author = "Hatcher, Allen",
	isbn = "0-521-79160-X; 0-521-79540-0",
	mrclass = "55-01 (55-00)",
	mrnumber = "1867354 (2002k:55001)",
	mrreviewer = "Donald W. Kahn",
	pages = "xii+544",
	publisher = "Cambridge University Press, Cambridge",
	title = "{Algebraic topology}",
	year = "2002"
}

@article{sem96,
	author = "Semmes, S.",
	coden = "SMATF6",
	doi = "10.1007/BF01587936",
	fjournal = "Selecta Mathematica. New Series",
	issn = "1022-1824",
	journal = "Selecta Math. (N.S.)",
	mrclass = "46E35 (42B20 53C23 57N15)",
	mrnumber = "1414889 (97j:46033)",
	mrreviewer = "Juha Heinonen",
	number = "2",
	pages = "155--295",
	title = "{Finding curves on general spaces through quantitative topology, with applications to {S}obolev and {P}oincar{\'e} inequalities}",
	url = "http://dx.doi.org/10.1007/BF01587936",
	volume = "2",
	year = "1996"
}

@article{res67,
	author = "Re\v{s}etnjak, Ju. G.",
	fjournal = "Akademija Nauk SSSR. Sibirskoe Otdelenie. Sibirski\u\i\ Matemati\v ceski\u\i\ \v Zurnal",
	issn = "0037-4474",
	journal = "Sibirsk. Mat. \v Z.",
	mrclass = "30.47",
	mrnumber = "0215990 (35 \#6825)",
	mrreviewer = "P. Caraman",
	pages = "629--658",
	title = "{Spatial mappings with bounded distortion}",
	volume = "8",
	year = "1967"
}

@book{ric93,
	author = "Rickman, Seppo",
	doi = "10.1007/978-3-642-78201-5",
	isbn = "3-540-56648-1",
	mrclass = "30C65",
	mrnumber = "1238941 (95g:30026)",
	mrreviewer = "Gaven J. Martin",
	pages = "x+213",
	publisher = "Springer-Verlag, Berlin",
	series = "{Ergebnisse der Mathematik und ihrer Grenzgebiete (3) [Results in Mathematics and Related Areas (3)]}",
	title = "{Quasiregular mappings}",
	url = "http://dx.doi.org/10.1007/978-3-642-78201-5",
	volume = "26",
	year = "1993"
}

@book{vai71,
	author = "V{\"a}is{\"a}l{\"a}, Jussi",
	mrclass = "30A60",
	mrnumber = "0454009 (56 \#12260)",
	mrreviewer = "F. W. Gehring",
	pages = "xiv+144",
	publisher = "Springer-Verlag, Berlin-New York",
	series = "{Lecture Notes in Mathematics, Vol. 229}",
	title = "{Lectures on {$n$}-dimensional quasiconformal mappings}",
	year = "1971"
}

@article{vai66,
	author = "V{\"a}is{\"a}l{\"a}, Jussi",
	journal = "Ann. Acad. Sci. Fenn. Ser. A I No.",
	mrclass = "55.60",
	mrnumber = "MR0200928 (34 \#814)",
	mrreviewer = "P. T. Church",
	pages = "10",
	title = "{Discrete open mappings on manifolds}",
	volume = "392",
	year = "1966"
}

@book{alm2000,
	author = "{Almgren Jr.}, Frederick J.",
	isbn = "981-02-4108-9",
	mrclass = "49-02 (35J20 49N60 49Q20 58E12)",
	mrnumber = "1777737",
	mrreviewer = "Brian Cabell White",
	note = "$Q$-valued functions minimizing Dirichlet's integral and the regularity of area-minimizing rectifiable currents up to codimension 2, With a preface by Jean E. Taylor and Vladimir Scheffer",
	pages = "xvi+955",
	publisher = "World Scientific Publishing Co., Inc., River Edge, NJ",
	series = "{World Scientific Monograph Series in Mathematics}",
	title = "{Almgren's big regularity paper}",
	volume = "1",
	year = "2000"
}

@article{del11,
	author = "{De Lellis}, Camillo and Spadaro, Emanuele Nunzio",
	doi = "10.1090/S0065-9266-10-00607-1",
	fjournal = "Memoirs of the American Mathematical Society",
	isbn = "978-0-8218-4914-9",
	issn = "0065-9266",
	journal = "Mem. Amer. Math. Soc.",
	mrclass = "49Q20 (35J50)",
	mrnumber = "2663735",
	mrreviewer = "Michele Miranda",
	number = "991",
	pages = "vi+79",
	title = "{{$Q$}-valued functions revisited}",
	url = "https://doi.org/10.1090/S0065-9266-10-00607-1",
	volume = "211",
	year = "2011"
}

@article{lui17,
	author = "Luisto, Rami",
	doi = "10.1007/s12220-016-9752-5",
	fjournal = "Journal of Geometric Analysis",
	issn = "1050-6926",
	journal = "J. Geom. Anal.",
	mrclass = "30L10 (30C65 57M12)",
	mrnumber = "3667422",
	mrreviewer = "David Matthew Freeman",
	number = "3",
	pages = "2081--2097",
	title = "{A characterization of {BLD}-mappings between metric spaces}",
	url = "https://doi.org/10.1007/s12220-016-9752-5",
	volume = "27",
	year = "2017"
}

\end{document}